\documentclass[a4paper,centertags,oneside,12pt]{amsart}
\usepackage{mathrsfs}
\usepackage{appendix}
\usepackage{amssymb}
\usepackage{fancyhdr}
\usepackage{charter}
\usepackage{typearea}
\usepackage{pdfsync}
\usepackage{color}
\usepackage[colorlinks,linkcolor=blue,anchorcolor=blue,citecolor=green]{hyperref}
\usepackage{mathrsfs}
\usepackage[a4paper,top=3cm,bottom=3cm,left=2.8cm,right=2.8cm]{geometry}

\usepackage{enumitem}
\usepackage{color}
\setlist[1]{itemsep=5pt}
\newcommand{\comment}[1]{}

\makeatletter
      \def\@setcopyright{}
      \def\serieslogo@{}
      \makeatother

\newcommand{\ddbar}{\overline\partial}
\newcommand{\p}{\partial}
\newcommand{\CC}{\mathbb C}
\newcommand{\pr}{\partial}

\newcommand{\ov}{\overline}
\newcommand{\sm}{\setminus}

\newcommand{\To}{\rightarrow}

\newtheorem{theorem}{Theorem}[section]
\newtheorem{lemma}[theorem]{Lemma}

\newtheorem{corollary}[theorem]{Corollary}

\newtheorem{proposition}[theorem]{Proposition}

\newtheorem*{conjecture}{Conjecture}

\newtheorem{remark}[theorem]{Remark}

\numberwithin{equation}{section}
\newtheorem*{theorem*}{Theorem}
\begin{document}

\title{Localization of Bergman Kernels and the Cheng--Yau Conjecture on Real Analytic Pseudoconvex Domains}
\author{Chin-Yu Hsiao}
\address{Department of Mathematics, National Taiwan University}
\email{chinyuhsiao@ntu.edu.tw}
\thanks{Chin-Yu Hsiao is partially supported by National Science and Technology Council project 113-2115-M-002-011-MY3.}
\author{Xiaojun Huang}
\address{Department of Mathematics, Rutgers University, New Brunswick, NJ 08903, USA.}
\email{huangx@math.rutgers.edu}
\thanks{Xiaojun Huang is partially supported  by NSF DMS-2247151}

\author{Xiaoshan Li}
\address{School of Mathematics and Statistics, Wuhan University, Wuhan, Hubei 430072, China.}
\email{xiaoshanli@whu.edu.cn}
\thanks{Xiaoshan Li is supported in part by NSFC (12361131577, 12271411)}
%\date{Feb 2026}
\maketitle 
%\begin{center}
%    \bf With an Appendix by Xiaojun Huang and Xiaoshan Li
%\end{center}
\bigskip

\begin{center}
\vspace{-8pt} \small \itshape
  Dedicated to Professor Ngaiming Mok on the occasion of his 70th birthday
\vspace{-5pt} 
\end{center}

\begin{abstract}
In this paper, we establish the localization of Bergman kernels for unbounded pseudoconvex domains near boundary points of finite D'Angelo type. This result was proved by Engliš more than twenty years ago for bounded pseudoconvex domains and had remained as an open question   in the unbounded setting.  Related foundational work was carried out by Fefferman, Kerzman, Boutet de Monvel--Sjöstrand, Boas, Bell, and others. Combining our localization theorem with an extension theorem of Mir--Zaitsev, we prove that the Bergman metric of a bounded pseudoconvex domain with real-analytic boundary is Einstein if and only if the domain is biholomorphic to the unit ball. This result advances a longstanding conjecture of Cheng and Yau.
A key step is to show that the Bergman metric of a smooth, possibly unbounded, pseudoconvex domain cannot be Kähler--Einstein if its boundary contains a non-strongly pseudoconvex \(h\)-extendible point. We further prove that a  bounded weakly pseudoconvex real-analytic domain    with a Kähler--Einstein Bergman metric must possess a weakly pseudoconvex \(h\)-extendible boundary point, thereby reducing the problem to the \(h\)-extendible setting.

This paper draws deeply on, and reveals essential connections among many subfields of microlocal analysis, several complex variables, and complex geometry.

\end{abstract}

\tableofcontents
\section{Introduction}
For a bounded domain $\Omega \subset \mathbb{C}^{n+1}$ with $n\ge 0$, the Bergman metric is a canonical K\"ahler metric that is invariant under biholomorphisms, reflecting the function-theoretic and geometric properties of the domain. Cheng and Yau \cite{CY80} established that every bounded pseudoconvex domain in $\mathbb{C}^{n+1}$ with a $C^2$--smooth boundary admits a unique complete K\"ahler--Einstein metric up to a scaling factor, which is also biholomorphically invariant. This  theorem was later generalized by Mok and Yau \cite{MY80}, who removed the boundary regularity assumption and proved the existence and uniqueness of such a metric for arbitrary bounded pseudoconvex domains. The K\"ahler--Einstein metric reflects the pluri-potential and  geometric property of the domain and is  established by solving  a complex Monge--Amp\`ere equation.

 A natural  problem arising from these works is to determine under what circumstances the Bergman metric and the complete K\"ahler--Einstein metric coincide. A classical conjecture of Yau \cite{Yau82} states that the Bergman metric of a bounded pseudoconvex domain  is complete and Einstein  if and only if the domain is biholomorphic to a bounded homogeneous domain. Earlier, Cheng \cite{C79} had conjectured a more specific characterization: the Bergman metric of a smoothly bounded strongly pseudoconvex domain is K\"ahler--Einstein if and only if the domain is biholomorphic to the complex unit ball.  An immediate consequence of the classical theorem of Wong \cite{W77} is that a smoothly bounded homogeneous domain is biholomorphic to the ball. Since the Bergman metric of a smoothly bounded pseudoconvex domain is complete \cite{Oh81}, combining conjectures of Cheng and Yau leads to the  following Cheng--Yau conjecture:

\begin{conjecture}[Cheng--Yau, \cite{C79,Yau82,W77}]\label{Cheng--Yau-Conjecture}A smoothly bounded pseudoconvex domain in $\CC^{n+1}$ is  Bergman--Einstein, that is, its Bergman metric is K\"ahler--Einstein, if and only if it is biholomorphic to the unit ball of the same dimension.
\end{conjecture}
Our first  main result of this paper is the resolution of this conjecture in the case of real analytic boundary. The case $n=0$ is a special case of the classical Qi-Keng Lu theorem \cite{Lu66},  and the case $n=1$ had been previously obtained by Savale and Xiao \cite{SX25}. 
\begin{theorem}\label{Main theorem 3}
    The Bergman metric of a  bounded real analytic  pseudoconvex domain in  $\CC^{n+1}$ with $n\ge 0$ is Einstein if and only if it is biholomorphic to the unit ball of the same dimension.
\end{theorem}

In fact, using a recent result of the last two authors, included in an appendix to the present paper, we prove a slightly different version of this theorem under the assumption that the scalar curvature is constant.

\renewcommand{\thetheorem}{1.1'}
\begin{theorem}\label{Main theorem-CSC}
    The Bergman metric of a bounded real analytic pseudoconvex domain in $\mathbb C^{n+1}$ with $n\geq 2$ has constant scalar curvature if and only if it is biholomorphic to the unit ball.
\end{theorem}
\renewcommand{\thetheorem}{\arabic{section}.\arabic{theorem}}

The main difficulty in resolving the above Cheng--Yau conjecture lies in the fact that both metrics are defined in a highly abstract manner and, apart from the case of the basic models, are essentially impossible to compute explicitly.

The Cheng--Yau conjecture remained dormant for almost twenty years before a breakthrough by Fu and Wong in their  paper \cite{FW97}. Their paper first relates the  Fefferman localization  of the Bergman kernel to the Einstein property of the Bergman metric, which has influenced  subsequent studies.  

Cheng's conjecture---also known as the Cheng--Yau conjecture for bounded strongly pseudoconvex domains---was proved in complex dimension two by Fu and Wong \cite{FW97} in the simply connected case and by Nemirovski and Shafikov \cite{NS06} in the general two diemsnional  case. Huang and Xiao \cite{HX21} subsequently resolved the conjecture in all dimensions, building on the work of many earlier authors.

In a recent advance, Savale and Xiao \cite{SX25} established the conjecture for smoothly bounded pseudoconvex domains of finite type in complex dimension two. An earlier result of Fu and Wong \cite{FW97} proved the same conclusion for smoothly bounded complete Reinhardt pseudoconvex domains of finite type in \(\mathbb{C}^2\). For a more comprehensive historical account and further references, we refer the reader to the recent paper of the last two authors with Scott James \cite{HJL25}.

Related variations of the Cheng--Yau conjecture have also been explored by S.~Li in \cite{Li05,Li09,Li16}, Gaussier--Zimmer \cite{GZ22} and in a recent paper by Yuan \cite{Yuan25}. There has also been extensive work on the Bergman geometry of bounded domains and more general complex manifolds. For a few representative references, we mention the celebrated work  of Mok  \cite{Mo89,Mo12} and the many references therein.

%More recently, Savale and Xiao \cite{SX25}  advanced the study of the Cheng--Yau Conjecture in the case of complex dimensional two. They proved that a smoothly bounded  pseudoconvex domain of finite type in ${\mathbb C}^2$ whose Bergman metric is Einstein must be biholomorphic to the unit ball. This result had been established earlier by Fu--Wong \cite{FW97} for smoothly bounded  complete Reinhardt pseudoconvex domains of finite type in ${\mathbb C}^2$. In a recent paper \cite{HJL25}, it is proved that the Bergman metric of a pseudoconvex domain in $\CC^{n}$ cannot be Einstein, if its boundary contains a non-smooth strongly pseudoconvex polyhedral point.

One of the  fundamental obstacles in addressing the Cheng--Yau conjecture has been the difficulty of localizing the analysis of the Bergman kernel  on unbounded pseudoconvex domains. 
Localization of Bergman kernels on smoothly bounded pseudoconvex  domains  was pioneered in earlier works by Kerzman \cite{K72}, Fefferman \cite{Fe74}, Boutet de Monvel and  Sj\"ostrand \cite{BS76}, Bell \cite{Be86}, Boas \cite{Bo87a, Bo87b}, Huang--Li \cite{HL23}, Hsiao--Marinescu \cite{HM25}, Ma--Marinescu \cite[Proposition 4.1.5]{MM07}, Marinescu--Savale \cite{MS24} among many others.
More than twenty--five  years ago, Engliš \cite{Eng01, Eng04} established a crucial localization principle for the Bergman kernel on bounded pseudoconvex domains near  smoothly boundary points of finite type in the sense of D'Angelo. Engliš also attempted to obtain an analogous localization result in the unbounded setting by a similar method. 
However, in \cite[Page 16]{G09}, examples are constructed which show that the key uniform estimate for the $\ov\partial$--Neumann operator, underlying Engliš's bounded case localization argument, already fails on the Siegel upper half-space.
Consequently, a localization theorem for unbounded pseudoconvex domains needs a very different approach and has remained open for the past twenty-five years, severely limiting the applicability of localization techniques in problems requiring precise boundary asymptotics on unbounded domains. 
A recent work of Ebenfelt-Xiao-Xu \cite{EXX25}  contains a localization  theorem, among many other results,  in the case of the unit disc bundle of a negatively curved holomorphic line bundle over a K\"ahler manifold.
%which suffices for their purposes. %(Notice that the setting in \cite{EXX25} is not in  a complex Euclidean space, so our Theorem \ref{Main theorem 1} does not include it as a special case.)
Nevertheless, a localization theorem for  unbounded pseudoconvex domains near a strongly pseudoconvex point or more generally  a D'Angelo finite type pseudoconvex boundary point, which is crucial for our proof of Theorems \ref{Main theorem 3} and \ref{Main theorem 2}, has remained  open.

Our second main result  in this paper  establishes a general localization theorem for the Bergman kernel on unbounded pseudoconvex domains near a D’Angelo finite type boundary point. This  settles the question left open by  Engliš after  his work  \cite{Eng01, Eng04}.
Our theorem provides the necessary analytic tool  for many new applications. Before stating our second main theorem, we review briefly  the history on the study of  the Bergman kernels.

Let $\Omega$ be a pseudoconvex domain in $\mathbb{C}^{n+1}$ with $n\ge 1$, and let $A^2(\Omega)$ denote the space of square--integrable holomorphic functions on $\Omega$. We denote by 
$B_{\Omega}^{(0)}: L^2(\Omega) \to A^2(\Omega)$ the Bergman projection, which is the orthogonal projection with respect to the standard Euclidean metric on $\mathbb{C}^{n+1}$. Its associated distribution kernel, $K_\Omega(x,y) \in \mathcal{D}'(\Omega \times \Omega)$, is the Bergman kernel. The study of the boundary behavior of $K_\Omega(x,y)$ is a classical and central theme in several complex variables, requiring a sophisticated interplay of microlocal analysis, harmonic analysis, and complex geometry. A seminal result in this field was established by Fefferman~\cite{Fe74}. For a bounded strictly pseudoconvex domain $\Omega = \{z \in \mathbb{C}^{n+1}:\, r(z) < 0\}$ with $r\in C^\infty(\overline\Omega)$ and  $dr \neq 0$ on $\partial \Omega$, Fefferman proved that the Bergman kernel on the diagonal admits the following expansion:
\begin{equation*}\label{feff}
K_\Omega(z,z) = a(z)(-r(z))^{-(n+2)} + b(z)\log(-r(z))
\end{equation*}
near the boundary, where $a, b \in C^\infty(\overline{\Omega})$, $a|_{\partial \Omega}\not =0$. Subsequently, in 1976, Boutet de Monvel and Sj\"ostrand~\cite{BS76} characterized the singularity of the full Bergman kernel $K_\Omega(z, w)$ by showing that it is a Fourier integral operator (FIO) with a complex phase. 
These results primarily concern bounded domains as their proofs were based on Kohn's sub-elliptic estimates of the $\overline\partial$--Neumann operator. 
In the case where $\Omega$ is an unbounded domain, significantly less is known about the properties of $K_\Omega(z, w)$. A natural question arises: can we characterize the boundary behavior of the Bergman kernel near strictly pseudoconvex points or pseudoconvex ponts of finite type on an unbounded pseudoconvex domain? This problem is of great  importance, as many questions in several complex variables necessitate a deep understanding of Bergman kernels on model domains, which are typically unbounded in nature. Another perspective on tackling this problem is through localization problems: Suppose there is another bounded domain $G$ whose boundary partially coincides with that of $\Omega$; do the Bergman operators  on the coinciding portion differ only by a smoothing operator (smooth up to the boundary)? This question  plays a significant role in  problems surrounding the aforementioned Cheng--Yau conjecture. Our second main result, which  answers this question, is stated as follows: 

\begin{theorem}\label{Main theorem 1}
Let $\Omega\subset\mathbb C^{n+1}$ be a possibly unbounded pseudoconvex domain ($n\ge 1$).
Let $p\in \p\Omega$ be a smooth boundary point of finite type in the sense of D'Angelo. Then for any neighborhood $\widetilde{U}$ of $p$ in $\CC^{n+1}$, there are a neighborhood $U$ of $p$ in $\CC^{n+1}$ with $ U\Subset \widetilde{U}$ and a smoothly bounded pseudoconvex domian $G\subset \Omega\cap \widetilde{U}$ of finite type in the sense of D'Angelo  such that $U\cap \Omega\subset  G$ and 
$$K_{\Omega}(z, w)-K_{G}(z, w)\in C^\infty((U\cap\overline G)\times (U\cap\overline G)).$$
In particular, if $p\in \p\Omega$ is  a smooth strongly pseudoconvex boundary  point of $\Omega$,  $G$ can be chosen to be strongly pseudoconvex and for $z(\in  G)$ near $p$, it holds that 
$$K_{\Omega}(z,z) = a(z)(-r(z))^{-(n+2)} + b(z)\log(-r(z)).$$
Here, $r(z)$ is a smooth defining function of $G$ with 
$r\in C^{\infty}(\overline G)$,\ $dr|_{\p G}\not =0$,  
$a(z), b(z)\in C^\infty(\overline G)$ and  $a(p)\not = 0.$
\end{theorem}

A crucial intermediate step in the proof of Theorem \ref{Main theorem 3} is to establish it first for  $h$--extendible domains. Let $\Omega$ be a  pseudoconvexdomain and let $p$ be a smooth boundary point of finite type in the sense of D’Angelo. By Catlin’s theory of multitype \cite{Ca84}, there exists a biholomorphically invariant, nondecreasing sequence of rational numbers $(m_0, m_1, \cdots, m_n)$, with  $m_0=1$ such that $m_{n-q+1}\leq\Delta_q$ for $1\leq q\leq n$, where  $\Delta_q$ denotes the $q$-type of  $\p\Omega$ at $p$ in the sense of D’Angelo \cite{D82}. When these equalities hold, we say that  $p$ is an $h$--extendible point \cite{Yu93, BSY95}. The advantage of the $h$--extendible property is that the local model of a pseudoconvex domain near such a point retains the precise geometric features of $\Omega$  near $p$. Partially using the Boas--Straube--Yu \cite{BSY95} dilation method, one can then reduce the analysis on $\Omega$ to its local model, which is of finite D’Angelo type with a real algebraic boundary. We  now state our third main  result:

\begin{theorem}\label{Main theorem 2}
Let \(\Omega \subset \mathbb{C}^{n+1}\) (\(n \geq 1\)) be a possibly unbounded pseudoconvex domain with a smooth, non-strongly pseudoconvex, \(h\)-extendible boundary point. Then the Bergman metric of \(\Omega\) cannot be Einstein. 
\end{theorem}

Similarly, we also have a CSC version of it:

\renewcommand{\thetheorem}{1.4'}
\begin{theorem}\label{Main theorem 2-CSC}
    Let \(\Omega \subset \mathbb{C}^{n+1}\) (\(n \geq 2\)) be a possibly unbounded pseudoconvex domain with a smooth, non-strongly pseudoconvex, \(h\)-extendible boundary point. Then the Bergman metric of \(\Omega\) cannot  have constant scalar curvature.
\end{theorem}
\renewcommand{\thetheorem}{\arabic{section}.\arabic{theorem}}

Since the $\Omega$ in Theorem \ref{Main theorem 2} is not assumed to be bounded, its Bergman metric may fail to exist at every  point of $\Omega$.  However, as we will explain in Section 3, since  $\Omega$ admits a smooth boundary point of finite D’Angelo type,  the Bergman metric is well defined on a certain non-empty open subset $\Omega^*$ of $\Omega$.
In the statement of Theorem \ref{Main theorem 2}, saying that the Bergman metric of  $\Omega$ is not Kähler--Einstein means that there is no open subset of  $\Omega^*$ on which the Bergman metric of $\Omega$ is Kähler--Einstein. Indeed, it is easy to show that $\Omega\sm \Omega^*$ is contained in a complex analytic hypersurface and thus $\Omega^*$ is connected.

$H$--extendible domains include a large class of weakly pseudoconvex bounded domains of finite D’Angelo type, in particular, all smoothly bounded convex domains of finite D'Angelo  type in $\CC^{n+1}$ and all weakly pseudoconvex finite D'Angelo type domains with at most one zero Levi-eigenvalue at each boundary point. As an immediate consequence, we obtain the following results:

\begin{corollary}\label{Main theorem 2-convex}
Let \(\Omega \subset \mathbb{C}^{n+1}\) (\(n \geq 1\)) be a smoothly bounded convex domain of finite D'Angelo type. Then the Bergman metric of \(\Omega\) is Einstein  if and only if \(\Omega\) is biholomorphic to the ball.
\end{corollary}

\renewcommand{\thetheorem}{1.6'}
\begin{corollary}\label{Main theorem 2-convex-CSC}
    Let \(\Omega \subset \mathbb{C}^{n+1}\) (\(n \geq 2\)) be a smoothly bounded convex domain of finite D'Angelo type. Then the Bergman metric of \(\Omega\)  has constant scalar curvature, if and only if \(\Omega\) is biholomorphic to the ball.
\end{corollary}
\renewcommand{\thetheorem}{\arabic{section}.\arabic{theorem}}

The paper is organized as follows:
In the next section, we first prove Theorem \ref{Main theorem 1}. We then proceed to prove Theorem \ref{Main theorem 2} and Theorem \ref{Main theorem 2-CSC}. 
To prove Theorem \ref{Main theorem 3}, it suffices to show that if $\Omega$ is weakly pseudoconvex with a real analytic boundary that  does not contain any non-trivial holomorphic curves, and if the Bergman metric  $\Omega$ is K\"ahler--Einstein, or even has constant scalar curvature,  then $\Omega$ admits a weakly pseudoconvex boundary point that is $h$--extendible. For this purpose, we first use Theorem \ref{Main theorem 1}, along with an extension  result of Mir--Zaitsev \cite{MZ21} (see also \cite{MMZ03, LMR23} and even earlier related results in \cite{Fo92, Fo89}, etc),  to deduce that $\partial \Omega$ can be locally proper holomorphically  mapped into the sphere. We then demonstrate that the boundary points where the map exhibits the simplest branching property are the weakly pseudoconvex $h$--extendible points that we are looking for.

In the literature, $h$--extendability was originally introduced to extend results such as the existence of peak functions from strongly pseudoconvex domains to certain domains of finite D’Angelo type, as an intermediate step towards the general case. To the best of our knowledge, the present work seems to be  the first in which $h$--extendability plays a central role as a bridge for treating general bounded real analytic domains. 

To conclude the introduction, we note a very recent preprint by Yuan Yuan \cite{Yuan26}. Building largely   on the analysis developed in the present paper---an earlier version of which appeared in \cite{HHL26}---as well as on our main Theorem~\ref{Main theorem 3}, and introducing many beautiful new ideas, Yuan resolves in \cite{Yuan26} the Cheng--Yau conjecture for smooth, possibly unbounded pseudoconvex Reinhardt domains.

\medskip

{\bf Acknowledgements}: The third author wishes to thank Emil Straube for bringing to his attention two papers by H. Boas, which are important for our present work. The second author expresses gratitude to Scott James for valuable discussions on $h$--extendible domains, to Siqi Fu and Bingyuan Liu for insightful comments on the exact regularity of Bergman projections.
%and to Yuan Yuan for many stimulating  discussions on the Cheng--Yau conjecture over the years.

\section{Localization of Bergman kernels on  unbounded pseudoconvex domains}
\subsection{Schwartz kernel theorem}
Let $G\Subset\mathbb C^{n+1}$ be a  bounded domain with a smooth boundary. Let $L^2(G)$ be the space of square integrable functions with a standard inner product and norm:
\begin{equation*}\label{e-gue260309yyd}
(f, g)_G:=\int_{G}f\overline g d\lambda,~ \|f\|_G^2:=(f,f)_G, ~\forall f, g\in L^2(G),\end{equation*}
where $d\lambda$ is the Lebesgue measure on $\mathbb C^{n+1}$. Similarly, for a domain $\Omega \subset \mathbb C^{n+1}$, we denote by $\|\cdot\|_{\Omega}$ and $(\cdot\,,\cdot)_{\Omega}$ the $L^2$ norm and inner product on $\Omega$, respectively.
We denote by $\mathcal D'(G)$ the space of distributions of $G$.
Let $s\geq0$ be a non-negative integer.  We denote by $W^s(G)$ the usual Sobolev space of order $s$ on $G$. That is, $$W^s(G)=\{u\in\mathcal D'(G): \partial_x^\alpha u\in L^2(G), |\alpha|\leq s\}$$ where $\alpha=(\alpha_1, \cdots, \alpha_{2n+2})$ is a multi-index,  $|\alpha|=\alpha_1+\cdots+\alpha_{2n+2}$, and $\partial_x^\alpha=\partial^{\alpha_1}_{x_1}\cdots\partial^{\alpha_{2n+2}}_{x_{2n+2}}$. The $W^s$-norm of $u\in W^s(G)$  is given by  $\|u\|_{s(G)}^2=\sum_{|\alpha|\leq s}\|\partial^\alpha u\|^2$. When $s=0$, $\|\cdot\|_{0(G)}$ is just the $L^2$--norm $\|\cdot\|_G$. It is well-known that $C^\infty(\overline G)$ is dense in $W^s(G)$ with respect to the $W^s$-norm. Furthermore, for any non-negative integer $s$, there exists a continuous linear operator 
\begin{equation}\label{e-gue260305yyd}
E_s: W^s(G)\rightarrow W^s(\mathbb C^{n+1}), \quad E_s u|_{G}=u.\end{equation} 
The extension operator $E_s$ can be chosen to be independent of $s$ (see \cite{St70}) and we sometimes write $E:=E_s$.  Let $W_0^s(G)$ be the completion of $C_0^\infty(G)$ with respect to the norm $\|\cdot\|_{s(G)}$. 
The dual of $W_0^s(G)$ is denoted by $W^{-s}(G)$ and let $\|\cdot\|_{-s(G)}$ denote the natural norm on $W^{-s}(G)$. 

We now introduce an alternative concept of duality, which is applicable in situations where test functions are not constrained to vanish on the boundary. Following the notation of \cite{Bo87a}, we denote by $W_{\ast}^{-s}(G)\subset\mathcal{D}'(G)$ the dual of the Hilbert  space $W^{s}(G)$, defined for $s \in \mathbb{N}$. 
The norm of an element $f \in W_{\ast}^{-s}(G)$ is defined as
\begin{equation*}
\|f\|^\ast_{-s(G)} 
:= \sup \bigl\{ | \langle g\,, f \rangle_G | : g \in C^{\infty}(\overline{G}),\; \|g\|_{s(G)} = 1 \bigr\},
\end{equation*}
where and in what follows, we use $\langle \cdot, \cdot \rangle_G$ to denote
the pairing on $G$ between a space and its dual. Similarly, we denote by $\langle \cdot, \cdot \rangle_{\mathbb C^{n+1}}$ the pairing on $\mathbb C^{n+1}$.

Recall that the generalized Schwarz inequality  for $f\in W^{-s}_{\ast}(G), g\in W^s(G)$ is given by 
$$|\langle g\,, f\rangle_G|\leq \|g\|_{s(G)} \|f\|^\ast_{-s(G)}.$$

Since $W^s(G)$ is a reflexive Hilbert space, we have
$$(W^s(G))^\ast=W^{-s}_{\ast}(G),~~(W^{-s}_\ast(G))^\ast=W^s(G).$$
Since $W_0^s(\mathbb C^{n+1})=W^s(\mathbb C^{n+1}), \forall s\in\mathbb N$, we have $W^{-s}_{\ast}(\mathbb C^{n+1})=W^{-s}(\mathbb C^{n+1})$. 

We can treat $C_0^\infty(G)$ as a subspace $ W_\ast^{-s}(G)$ in the following standard way. For each $\varphi\in C_0^\infty(G)$
     we can define $F_{\varphi}\in W_\ast^{-s}(G)$ by
    $$\langle  u, F_{\varphi}\rangle_G:=\int_{G}u\overline\varphi d\lambda=(u, \varphi)_G, \quad\forall u\in W^s(G).$$
    Since $|\int_{G}u\overline \varphi d\lambda|\leq \|u\|_G\cdot\|\varphi\|_G\leq \|u\|_{s(G)}\cdot\|\varphi\|_G$, thus $F_\varphi\in W^{-s}_\ast(G)$ and the map $C_0^\infty(G)\rightarrow W_\ast^{-s}(G), \quad \varphi\mapsto F_\varphi$, is injective.

From \eqref{e-gue260305yyd}, we can check that there is a constant $C_s>0$ such that for all $f\in C_0^\infty(G)$, we have
 $$\|f\|_{-s(\mathbb C^{n+1})}\leq \|f\|^\ast_{-s(G)}\leq C_s\|f\|_{-s(\mathbb C^{n+1})}.$$
Here and in what follows, we write $C$ or $C_s$ for constants that may be different in different context.

Now we define a restriction map $R: W^{-s}(\mathbb C^{n+1})\rightarrow W^{-s}_\ast(G)$ in the following sense:
\begin{equation*}
    \langle\varphi, Rv\rangle_G:=\langle  E\varphi, v\rangle_{\mathbb C^{n+1}}, \forall \varphi\in W^s(G)~\text{and}~ v\in W^{-s}(\mathbb C^{n+1}), 
\end{equation*}
where $E$ is the extension operator given above. Then
\begin{equation*}
    |\langle\varphi, Rv\rangle_G|\leq \|E\varphi\|_{s(\mathbb C^{n+1})}\cdot\|v\|_{-s(\mathbb C^{n+1})}^\ast\leq C_s \|\varphi\|_{s(G)}\cdot\|v\|_{-s(\mathbb C^{n+1})}.
\end{equation*}
Thus, $R$ is a continuous linear map. It is clear  that $$Rv=v, \forall v\in C_0^\infty(G).$$

\begin{lemma}\label{lem2.1}

\begin{enumerate}
    \item $C_0^\infty(G)$ is dense in $W_\ast^{-s}(G)$ for $s\in\mathbb N$.
    \item Furthermore, for each $u\in W_\ast^{-s}(G)$, there is a $\tilde u\in W^{-s}(\mathbb C^{n+1})$ such that $u=R\tilde u$ and $\|\tilde u\|_{-s(\mathbb C^{n+1})}\leq \|u\|_{-s(G)}^\ast$.
    \item For $g\in W^{-s}_\ast(G)\subset {\mathcal D}'(G)$ and any first order derivative $\partial_{x_j}, ~1\leq j\leq 2n+2$, we have $\partial_{x_j}g\in W^{-s-1}_\ast(G)$ and $\|\partial_{x_j} g\|^\ast_{(-s-1)(G)}\leq C_s \|g\|_{-s(G)}^\ast$ for some constant $C_s$.
\end{enumerate}
\end{lemma}
\begin{proof}
Let $T$ be any continuous linear functional on $W_\ast^{-s}(G)$. By Hahn--Banach theorem, to prove that $C_0^\infty(G)$ is dense in $W_\ast^{-s}(G)$ it suffices to prove that $T=0$ whenever $T|_{C_0^\infty(G)}=0$. Since $W^s(G)$ is a reflexive space for each $s\in\mathbb N$, the dual space of $W_\ast^{-s}(G)$ is $W^s(G)$. Thus, there exists a $u\in W^s(G)$ such that $T=u^{\ast\ast}$ and 
    $$\langle\phi, T\rangle_G=\langle\phi, u^{\ast\ast}\rangle_G=\langle u, \phi\rangle_G, ~\forall \phi\in  W_\ast^{-s}(G).$$
   By assumption, one has 
    $$0=\langle F_\varphi, T\rangle_G=\langle  u, F_\varphi\rangle_G=\int_{G}u\overline \varphi d\lambda, \forall \varphi\in C_0^\infty(G).$$
    It follows that $u=0$ and thus $T=0$. Hence, $C_0^\infty(G)$ is dense in $W_\ast^{-s}(G)$ for $s\in\mathbb N.$

    For any $u\in W_\ast^{-s}(G)$, by (1) of Lemma \ref{lem2.1}, there exists a sequence $\{u_n\}\subset C_0^\infty(G)$ such that $F_{u_n}\rightarrow u$ in $W_\ast^{-s}(G)$.
    Then \begin{equation*}
    \begin{split}
        \|F_{u_n}-F_{u_m}\|^\ast_{-s(G)}
        &=\sup\{|\langle  \varphi, F_{u_n}-F_{u_m}\rangle_G|: \varphi\in C^\infty(\overline G): \|\varphi\|_{s(G)}\leq 1\}\\
        &\geq \sup\{|\langle \varphi, F_{u_n}-F_{u_m}\rangle_{\mathbb C^{n+1}}|: \varphi\in C_0^\infty(\mathbb C^{n+1}), \|\varphi\|_{s(\mathbb C^{n+1})}\leq 1\}\\
        &=\|u_n-u_m\|_{-s(\mathbb C^{n+1})}.
        \end{split}
    \end{equation*}
    Thus, $\{u_n\}$ is a Cauchy sequence in $W^{-s}(\mathbb C^{n+1})$ and we assume $u_n\rightarrow \tilde u$ in $W^{-s}(\mathbb C^{n+1})$ and 
    \begin{equation*}
        \|\tilde u\|_{-s(\mathbb C^{n+1})}\leq \|u\|_{-s(G)}^\ast, \quad R\tilde u=u.
    \end{equation*}

  For $g\in W^{-s}_\ast(G)\subset\mathcal D'(G)$,  from (1) of Lemma \ref{lem2.1}, there exists a sequence $\{g_n\}\subset C_0^\infty(G)$ such that $g_n$ converges to $g$ in $W^{-s}_\ast(G)$. For any $\varphi\in W^{s+1}(G)$ and first order  differential operator $\partial_{x_j}$, we have the following 
    \begin{equation*}
        \langle \varphi, \partial_{x_j}g_n\rangle_G=\int_{G}\varphi\cdot\overline {\partial_{x_j}g_n} d\lambda=-\int_{G}  {\partial_{x_j}}\varphi\cdot\overline{g_n}  d\lambda.
    \end{equation*}
It follows that 
\begin{equation*}
    |\langle \varphi, \partial_{x_j}g_n\rangle_G|\leq \|g_n\|^\ast_{-s(G)}\cdot\|\partial_{x_j}\varphi\|_{s(G)}\leq C_s\|g_n\|^\ast_{-s(G)}\cdot\|\varphi\|_{s+1(G)}.
\end{equation*}
Hence, it follows that $\{\partial_{x_j}g_n\}$ is a Cauchy sequence in $W_\ast^{-s-1}(G)$ and $\partial_{x_j}g_n\rightarrow h$ in $W^{-s-1}_\ast(G)$. Moreover, $\partial_{x_j}g=h$ in the sense of distribution and $\|\partial_{x_j}g\|^\ast_{(-s-1)(G)}\leq C_s\|g\|^\ast_{-s(G)}$.
\end{proof} 
Recall that \(C_0^\infty(G)\) is dense in \(W_\ast^{-s}(G)\) and in \(L^2(G)\) with respect to the norms \(\|\cdot\|_{-s(G)}^\ast\) and \(\|\cdot\|\), respectively. Since  \(\|\cdot\|_{-s(G)}^\ast \leq \|\cdot\|\), it follows that every \(f \in L^2(G)\) can also be regarded as an element, denoted by $F_f$, of \(W_{\ast}^{-s}(G)\) via the pairing
\begin{equation*}
    \langle \varphi, F_f \rangle_G =\int_{G} \varphi \, \overline{f} \, d\lambda, \qquad \varphi \in W^s(G).
\end{equation*}
Also, we have a sequence $\{\phi_n\}\subset C^\infty_0(G)$ that converges to $f$ with respect to the $L^2$ norm and thus also with respect to the $\|\cdot\|^*_{-s(G)}$-norm.
Then $F_{\phi_n}$ converges to the element identified as above.  We next present the following version of the Schwarz kernel theorem:

\begin{lemma}\label{Boundary regularity}
    Let $G\Subset\mathbb C^{n+1}$ be a smooth bounded domain.  Let $P: C_0^\infty(G)\rightarrow \mathcal D'(G)$ be a continuous linear operator. We denote by $P(z, w)\in\mathcal D'(G\times G)$ the Schwarz kernel of $P$. If $P$ can be extended to a continuous operator $P: W^{-s}_\ast(G)\rightarrow W^s(G)$ for each $s\in\mathbb N$. Then $P(z, w)\in C^\infty(\overline G\times\overline G)$.
\end{lemma}
\begin{proof}
    We define a  linear operator  $\tilde P: W^{-s}(\mathbb C^{n+1})\rightarrow W^s(\mathbb C^{n+1})$ by
    \begin{equation*}
        \tilde Pu:= E\circ P\circ Ru, \quad\forall u\in W^{-s}(\mathbb C^{n+1}).
    \end{equation*}
    Since $E, P, R$ are continuous for each $s\in\mathbb N$, thus $\tilde P$ is continuous for each $s\in\mathbb N$. By the classical Schwartz kernel theorem, (see \cite{H90}, for instance), the Schwarz kernel of $\tilde P$ denoted by $\tilde P(z, w)$, is smooth on $\mathbb C^{n+1}\times\mathbb C^{n+1}$, that is, $\tilde P(z,w)\in C^\infty(\mathbb C^{n+1}\times\mathbb C^{n+1})$. On the other hand, for $u, v\in C_0^\infty(G)$,
    \begin{equation*}
    \begin{split}
       \langle  v(z)\otimes u(w), ~\tilde P(z, w)\rangle_{\mathbb C^{n+1}}
       &= \langle v,~\tilde Pu\rangle_{\mathbb C^{n+1}}=\langle v,~ E\circ P\circ Ru\rangle_{\mathbb C^{n+1}}=\langle v,~E\circ P u\rangle_{\mathbb C^{n+1}}\\
       &=\langle v,~Pu\rangle_G=\langle v(z)\otimes u(w),  ~P(z, w)\rangle_G.
       \end{split}
    \end{equation*}
    It follows that $\tilde P(z, w)=P(z, w)$ on $G\times G$. Hence, $P(z, w)\in C^\infty(\overline G\times\overline G)$.
\end{proof}
\subsection{Pseudolocal estimates for the $\overline\partial-$operator on finite-type domains}
In this section, we recall some results from \cite{Bo87b} that will be fundamental in the subsequent sections. Let $G \Subset \mathbb{C}^{n+1}$ be a bounded smooth pseudoconvex domain of finite type in the sense of D'Angelo.

Let 
$B_{G}^{(0)}: L^2(G)\rightarrow{\rm Ker\,}\ddbar$
be the Bergman projection from $L^2(G)$ onto the Bergman space of $L^2$--integrable holomorphic functions. 

The $\ddbar$-Neumann Laplacian on $(0, q)$-forms is then 
the non-negative self-adjoint densely defined  operator in the space 
$L^2_{(0,q)}(G)$:
\begin{equation*} \label{e-gue190312syd}
\Box_G^{(q)}=\ddbar\,\ddbar^*+\ddbar^*\,\ddbar: 
{\rm Dom\,}(\Box_G^{(q)})\subset L^2_{(0,q)}(G)\To L^2_{(0,q)}(G).
\end{equation*}
Since $G$ is bounded and pseudoconvex, both $\Box_{G}^{(0)}$ and $\Box_{G}^{(1)}$ have closed range in the corresponding $L^2$ spaces (see~\cite{Ho65, CS01}).
Consequently, the $\overline\partial$-Neumann operators $N_{G}^{(0)}: L^2(G)\rightarrow{\rm Dom}(\Box_{G}^{(0)})$ and $N_{G}^{(1)}: L^2_{(0,1)}(G)\rightarrow{\rm Dom}(\Box_{G}^{(1)})$ for $\Box^{(0)}_{G}$ and $\Box^{(1)}_{G}$, respectively, are well-defined. Based on the subelliptic estimates for the $\overline\partial$-Neumann problem (see Kohn \cite{Ko64, Ko63}, Kohn--Nirenberg \cite{KN65} and Catlin \cite{Ca87}), the operators $B_{G}^{(0)}$ and $N_{G}^{(1)}$ then satisfy the following pseudolocal estimates ( after applying a partition of unity to (6) and (8)  on \cite[Page 497]{Bo87b}).
\begin{lemma}[Boas, \cite{Bo87a} \cite{Bo87b}]
Let $G \Subset \mathbb{C}^{n+1}$ be a smoothly bounded  pseudoconvex domain of finite type in the sense of D'Angelo.  Let $\chi_1, \chi_2\in C^\infty(\overline G)$ with ${\rm supp}~\chi_1\cap{\rm supp}~\chi_2=\emptyset$. Then for each $s\in\mathbb N$,  there exists a constant $C_s$ such that
\begin{equation}\label{off diagonal of B}
        \|(\chi_1 B^{(0)}_{G}\chi_2)g\|_{s(G)} \leq C_s \|g\|^\ast_{-s(G)}, \quad \forall g \in C_0^\infty(G)
    \end{equation}
    and
\begin{equation}\label{strongl local estimate of N}
    \|(\chi_1 N_{G}^{(1)}\chi_2) f\|_{s(G)}\leq C_s\|f\|^\ast_{-s(G)}, \forall f\in C^\infty_{(0, 1)}(\overline G). 
\end{equation}
\end{lemma}
\subsection{$L^2$--estimates on unbounded pseudoconvex domains}
We next state a version of  H\"omander's theorem that  will be another crucial tool for our proof of Theorem \ref{Main theorem 1}:
\begin{theorem}[\cite{Ho65, GHH17, Hu}]\label{8-29-cor1}
  Let $\Omega\subset\mathbb C^{n+1}$ be a possibly unbounded pseudoconvex domain and let $\varphi\colon \Omega\to[-\infty,\infty)$ be a plurisubharmonic function.  
  Fix a boundary point $p\in\partial \Omega$ and suppose that the following conditions hold.
 \begin{enumerate}
    \item There exist an open neighborhood $\widehat U$ of $p$ in $\mathbb C^{n+1}$ and a constant $c>0$ such that $\varphi(z)-c|\,z\,|^{2}$ is plurisubharmonic on $\widehat U\cap\Omega$.
    \item $v\in L^{2}_{(0,1)}(\Omega,\varphi)$ is a $(0,1)$-form satisfying $\overline\partial v=0$ and $\operatorname{supp}v\subset \widehat U\cap \Omega$.
  \end{enumerate}
  Then there exists a  $u\in L^2(\Omega, \varphi)$  such that $\overline\partial u=v$ and
 \begin{equation*}\label{8-27-a3}
   \int_{\Omega}|u|^{2}e^{-\varphi}\,d\lambda
    \le\frac1{c}\int_{\Omega}|v|^{2}e^{-\varphi}\,d\lambda.
  \end{equation*}
\end{theorem}

Let $\Omega \subset \mathbb C^{n+1}$ be a possibly unbounded pseudoconvex domain, and let $p \in \partial \Omega$ be a smooth boundary point of finite type in the sense of D'Angelo.  
From the proof of Lemmas 8 and 9 in \cite{GHH17}, one can construct a bounded plurisubharmonic function $\psi$ on $\Omega$ such that $\psi-c|z|^2$ is  plurisubharmonic near $p$ for some constant $c>0$. 
Using this $\psi$ as a weight function and  applying Theorem \ref{8-29-cor1}, 
we obtain the following result. 
\begin{theorem}[\cite{HJL25}]\label{8-29-cor2}
Let $\Omega\subset\mathbb C^{n+1}$ be a possibly unbounded pseudoconvex domain. Let $p\in\partial \Omega$ be a smooth boundary point of finite type. Then there exists a small neighborhood $\widehat U$ of $p$ in $\mathbb C^{n+1}$ such that for any $f\in L^2_{(0, 1)}(\Omega)$ with $\overline\partial f=0$ and ${\rm supp }~f\subset \widehat U\cap {\Omega}$, there exists a solution $u\in L^2(\Omega)$ such that $\overline\partial u=f$ and $$\int_{\Omega} |u|^2d\lambda\leq C\int_{\Omega}|f|^2 d\lambda,$$
where the  constant $C$ does not depend on $f$.
\end{theorem}

\subsection{Proof of Theorem \ref{Main theorem 1}}

We now proceed to the proof of Theorem \ref{Main theorem 1}. The basic tools will be H\"ormander’s  $L^2$--estimates, the pseudo-local estimates of Kohn and Catlin, as well as the Schwartz kernel theorem discussed above.

Let \( \Omega \) be a possibly unbounded pseudoconvex domain in \( \mathbb{C}^{n+1} \). Let $B_\Omega^{(0)}$ be the Bergman projection of $\Omega$. We start with the following lemma which is an immediate consequence of Theorem \ref{8-29-cor2}.

\begin{lemma}\label{8-31-prop1}
Let $p\in\partial\Omega$ be a point of finite type in the sense of D'Angelo. Then there exist a neighborhood $\widehat U$ of $p$ in $\mathbb C^{n+1}$ and a constant $C>0$ such that
\begin{equation}\label{local closed range}
    \|(I-B_\Omega^{(0)})u\|^2_\Omega\leq C\|\overline\partial u\|^2_\Omega,\quad \text{for all } u\in C_c^\infty(\widehat U\cap\overline \Omega).
\end{equation}
\end{lemma}
\begin{proof}
Let $\widehat U$ be an open neighborhood of $p$ and let $C>0$ be the constant given in Theorem~\ref{8-29-cor2}. 
For $u\in C_c^\infty(\widehat U\cap\overline\Omega)$, set $v=\overline\partial u$. Then $v\in L^2_{(0,1)}(\Omega)$, $\overline\partial v=0$, and $\operatorname{supp} v\subset \widehat U\cap\Omega$. By Theorem~\ref{8-29-cor2}, there exists a solution $h\in L^2(\Omega)$ such that $\overline\partial h=\overline\partial u$ and 
    \begin{equation*}\label{8-29-a1}
    \|h\|_\Omega\leq \sqrt{C}\|\overline\partial u\|_\Omega.
    \end{equation*} 
Note that $(I-B_\Omega^{(0)})u=(I-B_\Omega^{(0)})h$ and $\|h\|^2_\Omega= \|(I-B_\Omega^{(0)})h\|^2_\Omega+ \|B_\Omega^{(0)}h\|_\Omega^2$. 
We have
    \begin{equation*}
    \begin{split}
        \|(I-B_\Omega^{(0)})u\|^2_\Omega\leq\|h\|^2_\Omega
        &\leq C\|\overline\partial u\|^2_\Omega.
        \end{split}
    \end{equation*}
\end{proof}
As before, let $\Omega$ be a pseudoconvex domain in $\CC^{n+1}$. Let $G\subset \Omega$ be a smoothly  bounded  pseudoconvex domain of finite type in the sense of D'Angelo. Let $K_{G}(z, w)$ and $K_{\Omega}(z, w)$ be the Bergman kernels of $G$ and $\Omega$, respectively. Assume that there exists a small open set $ U\subset\mathbb C^{n+1}$ with $U\cap\pr\Omega\neq\emptyset$ and $U\cap\overline\Omega=U\cap\overline G$.
We next   prove  the following theorem:
\begin{theorem}\label{t-gue251231ycdb}
With the same notations and assumptions  we just set up, we have
    \begin{equation}\label{12-8-a1}
        K_{\Omega}(z, w)-K_G(z, w)\in C^\infty((U\times U)\cap(\overline G\times\overline G)).
    \end{equation}
\end{theorem}

\begin{proof}
For any $p\in\partial \Omega\cap U$, since $p$ is of finite type in the sense of D'Angelo,  by Lemma \ref{8-31-prop1}, there exists a neighborhood $\widehat U\subset U$ of $p$ such that the statement in Lemma \ref{8-31-prop1} holds.
Choose neighborhoods $U_1, U_2$ of $p$ such that $\overline U_1\Subset U_2\subset\overline U_2\Subset \widehat U$. Choose cut-off functions $\chi_1\in C_0^\infty(U_2)$ with $\chi_1\equiv 1$ in a neighborhood of $\overline U_1$, $\chi_2\in C_0^\infty (\widehat U)$ with $\chi_2\equiv 1$ in a neighborhood of $\overline U_2$, and $\chi_3\in C_0^\infty(\widehat U)$ with $\chi_3\equiv1$ in a neighborhood of ${\rm supp}~\overline\partial\chi_2$ and ${\rm supp}~\chi_3\cap{\rm supp}~\chi_1=\emptyset$.
In the following, we use the notations
$(\cdot, \cdot)_{\Omega}$ and
 $\|\cdot\|_{\Omega}$ to denote the inner product and $L^2$--norm on $\Omega$, respectively. 
    By Lemma \ref{8-31-prop1}, for every $u\in C_0^\infty(G)$, one has
    \begin{equation}\label{12-3-a5}
    \begin{split}
        \|(I-B_\Omega^{(0)})(\chi_2B_{G}^{(0)}\chi_1u)\|_{\Omega}&\leq C\|\overline\partial(\chi_2 B_{G}^{(0)}\chi_1u)\|_{\Omega}
        =C\|(\overline\partial \chi_2) B_G^{(0)}(\chi_1 u)\|_\Omega\\
        &=C\|(\overline\partial \chi_2)\chi_3 B_G^{(0)}(\chi_1 u)\|_{G}
        \leq  C\|\chi_3 B_{G}^{(0)}\chi_1 u\|_{G}\\
        &= C\|r_1 u\|_{G}
        \leq C_s\|u\|^\ast_{-s(G)}, \forall s\in\mathbb N.
        \end{split}
    \end{equation}
Here  and in the rest of this section $C, C_s>0$  are constants which may be different in different contexts. Also we write here  $r_1=\chi_3 B_{G}^{(0)}\chi_1$. The last inequality in the above formula is deduced from (\ref{off diagonal of B}).
(\ref{12-3-a5}) provides the crucial initial estimate for our proof, as it combines Hörmander’s $L^2$--estimates (on $\Omega$) with the pseudo-local estimates of Kohn and Catlin (on $G$). 
The rest of the argument makes extensive use of these estimates, together with their Banach-space duals, to reach a point where the Schwartz kernel theorem applies.

Since $C_0^\infty(G)$ is dense in $W_\ast^{-s}(G)$ with respect to the norm $\|\cdot\|^\ast_{-s(G)}$, it follows that 
\begin{equation}\label{1-21-a3}
    (I-B_\Omega^{(0)})(\chi_2B_{G}^{(0)}\chi_1): W_\ast^{-s}(G)\rightarrow L^2(\Omega), \forall s\in\mathbb N, ~\text{is continuous}.
    \end{equation}

Since $G$ is bounded and pseudoconvex, we have the following Hodge decomposition:
 \begin{equation*}
     \begin{split}&\Box_G^{(0)}N_{G}^{(0)}+B_{G}^{(0)}=I ~\text{on}~L^2(G),\\
     &N_{G}^{(0)}\Box_{G}^{(0)}+B_{G}^{(0)}=I~\text{on}~{\rm Dom}(\Box_{G}^{(0)}),
     \end{split}
 \end{equation*}
 where we recall that $N_{G}^{(0)}$ is the $\overline\partial$-Neumann operator on $L^2$--integrable functions of $G$.
It follows that 
\begin{equation*}
\chi_2\Box_{G}^{(0)}N_{G}^{(0)}\chi_1 u+\chi_2 B_{G}^{(0)}\chi_1 u=\chi_1 u, \forall u\in C_0^\infty(G).
\end{equation*}
Hence,
\begin{equation}\label{1-21-a1}
    B_\Omega^{(0)}\chi_2\Box_{G}^{(0)}N_{G}^{(0)}\chi_1 u+B_{\Omega}^{(0)}\chi_2 B_{G}^{(0)}\chi_1 u=B_\Omega^{(0)}\chi_1 u.
\end{equation}
{\it We claim that}
    \begin{equation}\label{1-21-a2}
        B_\Omega^{(0)}\chi_2\Box_{G}^{(0)}N_{G}^{(0)}\chi_1: W_\ast^{-s}(G)\rightarrow L^2(\Omega), \forall s\in\mathbb N, ~\text{is continuous}.
    \end{equation}
Indeed,
for $u\in C_0^\infty(G)$, $v\in L^2(\Omega)$, we have
    \begin{equation*}
        \begin{split}(B_\Omega^{(0)}\chi_2\Box_{G}^{(0)}N_{G}^{(0)}\chi_1 u, v)_{\Omega}&=(\chi_2\Box_{G}^{(0)}N_{G}^{(0)}\chi_1 u, B_\Omega^{(0)}v)_{G}=(\Box_{G}^{(0)}N_{G}^{(0)}\chi_1u,\chi_2 B_\Omega^{(0)}v)_{G}\\
        &=(\overline\partial N_{G}^{(0)}\chi_1 u, (\overline\partial \chi_2) B_\Omega^{(0)}v)_{\Omega}=(B_\Omega^{(0)}(\overline\partial\chi_2)^{\wedge, \ast}\overline\partial N_G^{(0)}\chi_1 u, v)_{\Omega}.
        \end{split}
    \end{equation*}
    Thus, \begin{equation*}
        \begin{split}
        B_\Omega^{(0)}\chi_2\Box_{G}^{(0)}N_{G}^{(0)}(\chi_1 u)&=B_\Omega^{(0)}(\overline\partial\chi_2)^{\wedge, \ast}\overline\partial N_G^{(0)}(\chi_1 u)=B_\Omega^{(0)}(\overline\partial\chi_2)^{\wedge, \ast}N_G^{(1)}\overline\partial(\chi_1 u)\\
        &=B_\Omega^{(0)}(\overline\partial\chi_2)^{\wedge, \ast}(\chi_3N_{G}^{(1)}\chi_4)\overline\partial(\chi_1 u),
        \end{split}
    \end{equation*}
where $\chi_4\in C_0^\infty(\widehat U)$, $\chi_4|_{{\rm supp}~\chi_1}=1$ and ${\rm supp}~\chi_4\cap{\rm supp}~\chi_3=\emptyset$. Here, \( (\overline\partial\chi_2)^{\wedge, \ast} \) denotes the formal adjoint of the operator \( (\overline\partial\chi_2)^\wedge \) with respect to the pointwise Hermitian inner products on \( C^\infty(\Omega) \) and \( C^\infty_{(0, 1)}(\Omega) \) (both denoted by \( \langle\cdot\,,\cdot\rangle_e \)). The operator \( (\overline\partial\chi_2)^\wedge \) is defined by \( (\overline\partial\chi_2)^{\wedge} u := (\overline\partial\chi_2) u \) for \( u\in C^\infty(\Omega) \). Consequently, the adjoint is given by the following paring: 
\[
\langle (\overline\partial\chi_2)^{\wedge, \ast} g, u \rangle_e = \langle g, (\overline\partial\chi_2) u \rangle_e, \quad \forall u\in C^\infty(\Omega), \; g\in C^\infty_{(0, 1)}(\Omega).
\]
By the  pseudolocal estimate of $N_G^{(1)}$ in (\ref{strongl local estimate of N}), we have 
    \begin{equation*}
        \begin{split}
\|B_\Omega^{(0)}\chi_2\Box_{G}^{(0)}N_{G}^{(0)}(\chi_1 u)\|_\Omega
        &\leq \|(\overline\partial\chi_2)^{\wedge, \ast}(\chi_3 N_{G}^{(1)}\chi_4)\overline\partial(\chi_1 u)\|_G\leq C\|(\chi_3 N_G^{(1)}\chi_4)\overline\partial(\chi_1 u)\|_G\\
        &\leq C_s\|\overline\partial(\chi_1 u)\|^\ast_{(-s-1)(G)}\leq C_s\|u\|^\ast_{-s(G)}, \quad\forall u\in C_0^\infty(G).
        \end{split}
    \end{equation*}
    Since $C_0^\infty(G)$ is dense in $W_\ast^{-s}(G)$, thus we conclude  the claim.
    
    Set $$r_0u:=(\overline\partial\chi_2)^{\wedge, \ast}(\chi_3N_{G}^{(1)}\chi_4)\overline\partial(\chi_1 u), \forall u\in C_0^\infty(G).$$
     We have
    \[r_0: W_\ast^{-s}(G)\rightarrow W^s(G) ~\text{is continuous}~, \forall s\in\mathbb N.\]
   Thus, \begin{equation}\label{1-21-a6}
\begin{split}
&B_\Omega^{(0)}\chi_1-B_\Omega^{(0)}\chi_2 B_{G}^{(0)}\chi_1=B_\Omega^{(0)}r_0,\\
&r_0: W_\ast^{-s}(G)\rightarrow W^s(G) ~\text{is continuous}~, \forall s\in\mathbb N.
    \end{split}
\end{equation}

%\end{proof}
Following a beautiful original idea of Boutet de Monvel and Sjöstrand \cite{BS76}, we consider the Banach space adjoint (instead of  the Hilbert space adjoint)  $r_0^\ast$ of $r_0$:
 $$r_0^\ast: W_\ast^{-s}(G)\rightarrow W^s(G).$$ 
 Then
$$\langle r_0v, u\rangle =\langle v, r_0^\ast u\rangle
\hbox{   for } u, v\in W_\ast^{-s}(G)$$
and
\begin{equation}\label{1-24-a1}
    r_0^\ast: W_\ast^{-s}(G)\rightarrow W^s(G),~ \text{is continuous}, ~\forall s\in\mathbb N.
\end{equation}

Combining (\ref{1-21-a1}) and (\ref{1-21-a2}), we have 
    \begin{equation}\label{1-21-a4}
        B_\Omega^{(0)}\chi_1-B_\Omega^{(0)}(\chi_2 B_G^{(0)}\chi_1): W_\ast^{-s}(G)\rightarrow L^2(\Omega)~\text{is continuous},~\forall s\in\mathbb N.
    \end{equation}
Combining (\ref{1-21-a3}) and (\ref{1-21-a4}), we have the following
\begin{equation}\label{1-21-a5}
    B_\Omega^{(0)}\chi_1-\chi_2 B_G^{(0)}\chi_1: W_\ast^{-s}(G)\rightarrow L^2(\Omega) ~\text{is continuous}, \forall s\in\mathbb N.
\end{equation}
Since $G$ is a smoothly bounded pseudoconvex domain of finite D'Angelo type, we claim that $B_G^{(0)}$ is exact regular on $W^s(G)$ and $W^{-s}_\ast(G)$, respectively, for $s\in\mathbb N\cup\{0\}$. That is, both $B_G^{(0)}: W^s(G)\rightarrow W^s(G)$ and $B_G^{(0)}: W^{-s}_\ast(G)\rightarrow W^{-s}_\ast(G)$ are continuous for $s\in\mathbb N\cup\{0\}$. Indeed, by the work of Kohn \cite{Ko64, Ko63}, Kohn-Nirenberg \cite{KN65} and Catlin \cite{Ca87},
$B_G^{(0)}: W^s(G)\rightarrow W^s(G)$
is continuous for $s\geq 0$. For $u\in C_0^\infty(G), v\in W^s(G)$ we have
\begin{equation*}
    |(B_G^{(0)} u, v)|=|(u, B_G^{(0)}v)|=|\langle u, B_G^{(0)}v\rangle_G|\leq \|u\|_{-s(G)}^\ast\cdot\|B_G^{(0)}v\|_{s(G)}\leq C_s\|u\|_{-s(G)}^\ast \|v\|_{s(G)}.
\end{equation*}
Thus,
\begin{equation*}
    \|B_G^{(0)}u\|_{-s(G)}^\ast\leq C_s\|u\|_{-s(G)}^\ast, \quad\forall u\in C_0^\infty(G).
\end{equation*}
Since the $C_0^\infty(G)$ is dense in $W_{\ast}^{-s}(G)$, then $B_G^{(0)}$ can be extended to $W^{-s}_\ast(G)$ continuously and  we get the conclusion of the claim.
Thus, it follows from (\ref{1-21-a5}) that 
\begin{equation}\label{1-21-b5}
    B_\Omega^{(0)}\chi_1: W_\ast^{-s}(G)\rightarrow W_\ast^{-s}(G)~\text{is continuous}~, \forall s\in\mathbb N. 
\end{equation}
It follows from (\ref{1-24-a1}) that
\begin{equation}\label{2-18-a1}
    r_0^\ast B_\Omega^{(0)}\chi_1: W_\ast^{-s}(G)\rightarrow W^s(G)~\text{is continuous}~, \forall s\in\mathbb N.
\end{equation}
That is, 
\begin{equation*}
    \|r_0^\ast B_\Omega^{(0)}\chi_1 u\|_{s(G)}\leq C_s\|u\|_{-s(G)}^\ast, \forall u\in W_\ast^{-s}(G).
\end{equation*}
Taking  the Banach space adjoint of $r_0^\ast B_\Omega^{(0)}\chi_1$, we  have
\begin{equation*}
    \chi_1 B_\Omega^{(0)}r_0: W_\ast^{-s}(G)\rightarrow W^s(G)~\text{is continuous}~, \forall s\in\mathbb N.
\end{equation*}
Indeed, for $v\in C_0^\infty(G)$ and $u\in W_\ast^{-s}(G)$, we have $\chi_1 B_\Omega^{(0)}r_0 u\in L^2(G)$ and 
\begin{equation}\label{2-18-a2}
\begin{split}
    (\chi_1 B_\Omega^{(0)}r_0 u, v)_G&=(B_\Omega^{(0)}r_0 u, \chi_1v)_G=(B_\Omega^{(0)}r_0 u, \chi_1v)_\Omega
    =(r_0 u, B_\Omega^{(0)}(\chi_1v))_\Omega\\
    &=(r_0 u, B_\Omega^{(0)}(\chi_1v))_G
    =\langle r_0 u, B_\Omega^{(0)}\chi_1 v\rangle_G=\langle u, r_0^\ast B_\Omega^{(0)}\chi_1v\rangle_G.
    \end{split}
\end{equation}
It follows from (\ref{2-18-a1}) and (\ref{2-18-a2}) that
\begin{equation*}
    |(\chi_1 B_\Omega^{(0)}r_0 u, v)_G|\leq \|u\|_{-s(G)}^\ast \cdot\|r_0^\ast B_\Omega^{(0)}\chi_1 v\|_{s(G)}\leq C_s\|u\|_{-s(G)}^\ast\|v\|_{-s(G)}^\ast, \forall v\in C_0^\infty(G).
\end{equation*}
By the density of $C_0^\infty(G)$ in $W_{\ast}^{-s}(G)$, we have that $$\chi_1 B_\Omega^{(0)}r_0 u\in (W^{-s}_\ast(G))^\ast=W^s(G)$$ and 
\begin{equation*}
    \|\chi_1 B_\Omega^{(0)}r_0 u\|_{s(G)}\leq C_s\|u\|_{-s(G)}^\ast.
\end{equation*}
It now follows from (\ref{1-21-a6}) that
\begin{equation}\label{1-21-b9}
    \chi_1 B_\Omega^{(0)}\chi_1-\chi_1 B_\Omega^{(0)}(\chi_2 B_G^{(0)}\chi_1): W_\ast^{-s}(G)\rightarrow W^s(G)~\text{is continuous}~, \forall s\in\mathbb N.
\end{equation}
Taking the adjoint of (\ref{1-21-a5}), one has
\begin{equation}\label{1-21-a7}
    \chi_1 B_\Omega^{(0)}-\chi_1 B_G^{(0)}\chi_2: L^2(\Omega)\rightarrow  W^s(G)~\text{is continuous}~, \forall s\in\mathbb N.
\end{equation}
Indeed, for $u\in L^2(\Omega)$, $v\in C_0^\infty(G)$ we have $[\chi_1 B_\Omega^{(0)}-\chi_1 B_G^{(0)}\chi_2] u\in L^2(G)$ and 
\begin{equation}\label{2-18-a5}
\begin{split}
    (\chi_1 B_\Omega^{(0)}u-\chi_2 B_G^{(0)}\chi_1 u, v)_G
    &=(B_\Omega^{(0)}u, \chi_1 v)_G-(B_G^{(0)}\chi_1 u, \chi_2 v)_G\\
    &=(u, B_\Omega^{(0)}\chi_1 v-\chi_1 B_G^{(0)}\chi_2 v)_G
    \end{split}
\end{equation}
It follows from (\ref{1-21-a5}) and (\ref{2-18-a5}) that 
\begin{equation*}
    |(\chi_1 B_\Omega^{(0)}u-\chi_2 B_G^{(0)}\chi_1 u, v)_G|\leq C_s\|u\|\cdot\|v\|_{-s(G)}^\ast.
\end{equation*}
Hence, $$\chi_1 B_\Omega^{(0)}u-\chi_1 B_G^{(0)}\chi_2 u\in W^s(G)$$ and 
\begin{equation*}
    \|(\chi_1 B_\Omega^{(0)}-\chi_1 B_G^{(0)}\chi_2 )u\|_{s(G)}\leq C_s\|u\|, \forall u\in L^2(\Omega).
\end{equation*}

We write (\ref{1-21-a5}) and (\ref{1-21-a7}) together,
\begin{equation*}
\begin{split}
    &B_\Omega^{(0)}\chi_1- \chi_2B_{G}^{(0)}\chi_1: W_\ast^{-s}(G)\rightarrow L^2(\Omega), ~\text{is continuous}, \forall s\in\mathbb N,\\
    &\chi_1 B_\Omega^{(0)}-\chi_1 B_{G}^{(0)}\chi_2: L^2(\Omega)\rightarrow W^s(G), ~\text{is continuous}, \forall s\in\mathbb N.
    \end{split}
\end{equation*}
We consider the composition of the above operators:

\begin{equation}\label{1-21-b10}
    (\chi_1 B_\Omega^{(0)}-\chi_1 B_{G}^{(0)}\chi_2)(B_\Omega^{(0)}\chi_1-\chi_2 B_{G}^{(0)}\chi_1): W_\ast^{-s}(G)\rightarrow W^s(G)~\text{is continuous}, \forall s\in\mathbb N.
\end{equation}

By a direct calculation, and  applying  Lemma~\ref{Boundary regularity} and \eqref{1-21-b9}, we have 
\begin{equation}\label{1-21-b8}
    \begin{split}
       &(\chi_1 B_\Omega^{(0)}-\chi_1 B_{G}^{(0)}\chi_2)(B_\Omega^{(0)}\chi_1-\chi_2 B_{G}^{(0)}\chi_1)\\
       &=\chi_1 B_\Omega^{(0)}\chi_1-\chi_1 B_\Omega^{(0)}\chi_2 B_G^{(0)}\chi_1-\chi_1 B_G^{(0)}\chi_2 B_\Omega^{(0)}\chi_1+\chi_1 B_G^{(0)}\chi^2_2B_{G}^{(0)}\chi_1\\
       &=-\chi_1 B_\Omega^{(0)}\chi_1+\chi_1 B_G^{(0)}\chi^2_2 B_{G}^{(0)}\chi_1+F\\
       &=-\chi_1 B_\Omega^{(0)}\chi_1+\chi_1B_{G}^{(0)}\chi_1+\chi_1 B_G^{(0)}(\chi^2_2-1)B_{G}^{(0)}\chi_1+F,
 \end{split}
\end{equation}
where $$F=\left(\chi_1 B_\Omega^{(0)}\chi_1-\chi_1 B_\Omega^{(0)}\chi_2 B_G^{(0)}\chi_1\right) + \left(\chi_1 B_\Omega^{(0)}\chi_1-\chi_1 B_G^{(0)}\chi_2 B_\Omega^{(0)}\chi_1\right):=F_1+F_2.$$
By  (\ref{1-21-b9}), $F_1$ is a continuous map from
 $W_\ast^{-s}(G)$ into $ W^s(G)$ $\forall s\in\mathbb N$. Since $F_2$ is the Banach space adjoint of $F_1$, 
 we conclude that $F_2$ is a continuous map and, thus, $F$  is also a continuous map from
% $F_2$ and, thus, $F$ are also continuous maps from
 $W_\ast^{-s}(G)$ into $ W^s(G)$ $\forall s\in\mathbb N$. Another way to see that $F_2$ is continuous from $W_{\ast}^{-s}(G)$ to $W^s(G)$ is to note the following computation:
\begin{equation*}
\begin{split}
    F_2&=\chi_1 B_\Omega^{(0)}\chi_1-\chi_1 B_G^{(0)}\chi_2 B_\Omega^{(0)}\chi_1\\
    &=\chi_1 B_\Omega^{(0)}\chi_1-\chi_1 B_G^{(0)}(\chi_2-1)B_\Omega^{(0)}\chi_1-\chi_1 B_G^{(0)}B_\Omega^{(0)}\chi_1\\
    &=-[\chi_1 B_G^{(0)}(\chi_2-1)]B_\Omega^{(0)}\chi_1.
    \end{split}
\end{equation*}
It follows from (\ref{off diagonal of B}) that 
\begin{equation}\label{1-21-b7}
    \chi_1 B_G^{(0)}(\chi_2^2-1)B_G^{(0)}\chi_1: W_\ast^{-s}(G)\rightarrow W^s(G)~\text{is continuous}, \forall s\in\mathbb N.
\end{equation}
From \eqref{1-21-b10}, \eqref{1-21-b8} and \eqref{1-21-b7}, we thus  deduce that  
\begin{equation*}
    \chi_1 B_G^{(0)}\chi_1-\chi_1 B_\Omega^{(0)}\chi_1: W_\ast^{-s}(G)\rightarrow W^s(G)~\text{is continuous}, \forall s\in \mathbb N.
\end{equation*}
Since the Schwartz kernel of $\chi_1 B_{G}^{(0)}\chi_1-\chi_1 B_\Omega^{(0)}\chi_1$ is given by $$\chi_1 (z)K_G(z, w)\chi_1(w)-\chi_1 (z)K_{\Omega}(z, w)\chi_1(w),$$  by Lemma \ref{Boundary regularity}, we have 
$$\chi_1 (z)K_G(z, w)\chi_1(w)-\chi_1 (z)K_{\Omega}(z, w)\chi_1(w)\in C^\infty(\overline G\times\overline G).$$
Thus $K_G(z, w)-K_{\Omega}(z, w)\in C^\infty((U_1\cap\overline G)\times (U_1\cap\overline G))$. 
Since $p\in\partial \Omega\cap U$ is arbitrary, we conclude the proof of Theorem \ref{t-gue251231ycdb}.
\end{proof}

\begin{proof}[Proof of Theorem \ref{Main theorem 1}:] To conclude the proof of Theorem \ref{Main theorem 1}, we only need to combine Theorem  \ref{t-gue251231ycdb}  with the following lemma of Bell and the asymptotic expansion of Fefferman in the strongly pseudoconvex case \cite{Fe74}. The  proof of the Bell lemma  can be found in \cite[Section 4]{Be86}. (Although the domain $\Omega$ is assumed to be bounded in \cite{Be86}, the proof remains valid for unbounded pseudoconvex domains without modification.)
\begin{lemma}[\cite{Be86}]\label{subdomian}
Let $\Omega\subset\mathbb C^{n+1}$ be a possibly unbounded domain and let $p\in\partial \Omega$ be a smooth boundary point of finite type in the sense of D'Angelo. Then there exist a smooth bounded domain $G\subset \Omega$ of finite type and a neighborhood $U$ of $p$ in $\mathbb C^{n+1}$ such that
$U\cap\overline \Omega = U\cap\overline G$. In particular, if $p\in\partial \Omega$ is a strongly pseudoconvex boundary point, then $G$ can be chosen to be strongly pseudoconvex.
\end{lemma}
\end{proof}
In the proof of Theorems \ref{Main theorem 3}  and \ref{Main theorem 2},  Theorem \ref{Main theorem 1} is crucially used  through its  following derivation:  
\begin{theorem}\label{spherical}
    Let $\Omega\subset\mathbb C^{n+1}$ be a possibly unbounded pseudoconvex domain with $p\in\partial \Omega$ a strongly pseudoconvex smooth boundary point. Assume that the Bergman metric of $\Omega$ is Einstein in the subdomain $\Omega^*$ of $\Omega$ where the Bergman metric is well-defined. Then $\partial \Omega$ is spherical near $p$. Namely, a small open piece of 
$\partial\Omega$ is CR-diffeomorphic to an open piece of the boundary of the unit ball $\mathbb{B}^n$. 
\end{theorem}
\begin{proof}
By Lemma \ref{subdomian}, there exists  a bounded strongly pseudoconvex domain $G\subset \Omega$  and a neighborhood $U$ of $p$ such that $\overline G\cap U=\overline\Omega\cap U$. To prove the theorem, it suffices to show that $\partial G$ is spherical near $p$. Given the localization of the Bergman kernel for $\Omega$ established in Theorem \ref{Main theorem 1}, the same argument as in the proof of Theorem 2.1 in \cite{HL23} directly yields this result.
\end{proof}
Indeed, in the appendix, we will use Theorem~\ref{Main theorem 1} to prove a stronger version Theorem  \ref{main thm-appendix} of the preceding theorem when $n\ge 2$ under the assumption that the Bergman metric has constant scalar curvature.

%\begin{theorem} \label{main %thm-appendix-001}
%Let $\Omega
%\subset {\mathbb C}^{n+1}$ with $n\ge 2$ be a possibly unbounded pseudoconvex domain with $p\in \partial\Omega$ a  smooth strongly pseudoconvex  boundary point of $\Omega$. If $\Omega$  has  a { CSC} Bergman metric then $\partial\Omega$ is locally spherical near $p$.
%\end{theorem}

\begin{remark}

The localization result for Bergman kernels in Theorem \ref{t-gue251231ycdb} extends  to domains in a complex manifold with analogous properties. 
Let $\widetilde\Omega$  be an $n$-dimensional Hermitian manifold, and let  $\Omega$  be a subdomain with a smooth  boundary point  $p\in \p\Omega\subset \widetilde{\Omega}$.
Let $G$  be a smoothly bounded pseudoconvex domain of finite D'Angelo type in   $\widetilde{\Omega}$ such that, in a small neighborhood  of  $U$ of $p$, we have $U\cap\overline\Omega = U\cap\overline G$, and such that  $G$ is contained in a local holomorphic chart.
		The Bergman kernel  $K_\Omega$ is the reproducing kernel for the Bergman space $A^2_{(n,0)}(\Omega)$, consisting of $L^2$--integrable holomorphic $(n, 0)$-forms on $\Omega$. The Bergman projection   $B_\Omega^{(n,0)}$ is the orthogonal projection from  $L^2_{(n,0)}(\Omega)$ onto $A^2_{(n,0)}(\Omega)$.
		The key estimate enabling the localization is the following $\overline\partial$-estimate near $p$: there exists a neighborhood 
      $\widehat U$ of $p$ in $\widetilde\Omega$ and a constant $C>0$ such that 
\[
\|(I - B_\Omega^{(n,0)})u\|_\Omega \leq C \|\overline\partial u\|_\Omega, \qquad \forall u \in C_c^\infty(\widehat U\cap\overline\Omega; \Lambda^{n,0}),
\]
 where $\Lambda^{n,0}$ denotes the bundle of $(n,0)$-forms over $\widetilde\Omega$, and the norms are taken with respect to the given Hermitian metric on $\widetilde\Omega$. Once this estimate is established, the proof proceeds verbatim to yield the same localization result as in Theorem \ref{t-gue251231ycdb}. This is the case, for instance, when $\Omega$ is Stein.
\end{remark}

\section{Bergman--Einstein metrics on $h$--extendible domains}
Let $\Omega\subset\mathbb C^{n+1}$ be a  pseudoconvex domain. Let $p\in\partial \Omega$ be a smooth finite type boundary point in the sense of D'Angelo. According to Catlin's theory of multitype \cite{Ca84}, there is a biholomorphically invariant nondecreasing sequence of rational numbers $(m_0, m_1, \cdots, m_n)$, with $m_0=1$, such that $m_{n-q+1}\leq\Delta_q$ for $1\leq q\leq n$, where $\Delta_q$ is the $q$-type at $p$ in the sense of D'Angelo \cite{D82}. In a suitable coordinate system  $(z_0, z')=(z_0, z_1, \cdots, z_n)$ centered at $p$, there exists a real-valued, plurisubharmonic polynomial $P$ with no harmonic terms such that $\Omega$ is locally defined near $p$ by
$${\rm Re}~z_0+P(z')+o(\sum_{j=0}^n|z_j|^{m_j})<0.$$
Here, we assign the weight $z_j$ to be $\frac{1}{m_j}$ for each $j$ and $P$ is then a  weighted homogeneous polynomial  in the sense that  $P(\pi_t(z'))=tP(z')$ where $\pi_t$ is the anisotropic dilation given by
$\pi_t(z)=(tz_0, t^{\frac{1}{m_1}}z_1, \cdots, t^{\frac{1}{m_n}}z_n)$. The unbounded domain $D=\{(z_0, z'):{\rm Re~z_0}+P(z')<0\}$ is called a local model for $\Omega$ at $p$. When $m_{n-q+1}=\Delta_q<\infty$ for $1\leq q\leq n$, we say that $\Omega$ is $h$--extendible at $p$. 
It was proved  in  \cite{ Yu93, Yu94} that $\Omega$ is h-extendible at $p$ if and only if the local model $D$ admits a bumping function $a(z')$ with the following properties:
\begin{enumerate}
    \item on $\mathbb C^n\setminus\{0\}$, the function $a$ is $C^\infty$-smooth and positive;
    \item $a$ is weighed homogeneous in the same sense as for  $P$;
    \item $P(z')-\varepsilon a(z')$ is strictly plurisubharmonic on $\mathbb C^n\setminus\{0\}$ when $0<\varepsilon\leq 1$.
\end{enumerate}
Now assume that $p\in \p\Omega$ is $h$--extendible. By \cite{BSY95, Yu93, Yu94}, we may choose local holomorphic coordinates $(z_0, z')$ centered at $p$ such that $\Omega$ is defined by $r<0$ with $$r(z_0, z')={\rm Re}~z_0+P(z')+O(\sigma(z')^{1+\alpha})+O(|{\rm Im}~z_0|^{2})$$
where $0<\alpha<1$ is a certain  positive constant and $\sigma(z')=\sum_{j=1}^n|z_j|^{m_j}$. 
These normalized coordinates and the corresponding local model $D$ are fixed  in this section from now on.  

Still write  $A^{2}(\Omega)$  for its Bergman space consisting of holomorphic functions on $\Omega$  that are square-integrable with respect to the Lebesgue measure.    
Let $\{\varphi_{j}\}_{j=1}^\infty$ be an orthonormal basis for  $A^{2}(\Omega)$   with respect to the standard inner-product.
The Bergman kernel  function  of   $\Omega$ is then defined by:
\[
K_{\Omega}(z,{z})=\sum_{j=1}^{\infty}\varphi_{j}(z)\overline{\varphi_{j}(z)},\quad\forall z\in \Omega.
\]
The Bergman metric $g_{\Omega}$ is well defined on an  open subset $\Omega^\ast\subset\Omega$ that contains a one-sided neighborhood of $p$ in $\Omega$. For simplicity of notation, we write $\Omega^*$ for the largest open subset of $\Omega$, where $g_{\Omega}$ is well defined (see \cite{HJL25}).

 When $\Omega$ is a bounded domain or an $h$--extendible model then $\Omega^\ast=\Omega$  by a result of Boas--Straube--Yu \cite{BSY95}.
On $\Omega^\ast$, the Bergman metric is given by
\[
g_{\Omega}=\sum_{i,j=0}^{n}g_{i\overline{j}}\,dz_{i}\otimes d\overline{z}_{j},\quad\text{where}\quad g_{i\overline{j}}=\frac{\partial^{2}\log K_{\Omega}}{\partial z_{i}\partial\overline{z}_{j}};
\]
and the  Bergman  norm  is defined as
\[
g_{\Omega}(z,u)=\biggl(\sum_{i,j=0}^{n}g_{i\overline{j}}(z)\,u_{i}\overline{u}_{j}\biggr)^{1/2},\quad\forall u\in\mathbb{C}^{n+1}.
\]

The Bergman canonical invariant function  is a biholomrphic invariant and is a positive real analytic function defined over $\Omega^*$ by
\[
J_{\Omega}(z):=\frac{\det G_{\Omega}(z)}{K_{\Omega}(z,\overline{z})},\quad\text{where}\quad G_{\Omega}(z)=\bigl(g_{i\overline{j}}(z)\bigr).
\]
The Ricci curvature tensor of the Bergman metric $g_{\Omega}$  is given by
\[
R_{\Omega}=\sum_{\alpha,\beta=0}^{n}R_{\alpha\overline{\beta}}\,dz_{\alpha}\otimes d\overline{z}_{\beta}\quad\text{with}\quad R_{\alpha\overline{\beta}}=-\frac{\partial^{2}\log\det G_{\Omega}}{\partial z_{\alpha}\partial\overline{z}_{\beta}},
\]
and the Ricci curvature along the  direction $u\in\mathbb{C}^{n}\setminus\{0\}$ is given by
\[
R_{\Omega}(z,u)=\frac{\sum_{\alpha,\beta=0}^{n}R_{\alpha\overline{\beta}}\,u_{\alpha}\overline{u}_{\beta}}{g_{\Omega}(z,u)^{2}}.
\]
The scalar curvature of the Bergman metric is the trace of the Ricci tensor with respect to the metric:
$$ 
S_{\Omega}=g^{i\bar{j}} \,\mathrm{Ric}_{i\bar{j}}=
-\,g^{i\bar{j}} \,\frac{\partial^2}{\partial z^i \,\partial \bar{z}^j}
\log \det\bigl(g_{k\bar{\ell}}\bigr).
$$ 
The Bergman metric is said to have constant scalar curvature on $\Omega^\ast$ if there exists a constant $c$ such that $$S_\Omega\equiv c.$$ In this case, we call it a CSC Bergman metric.
In particular, 
the Bergman metric $g_{\Omega}$  is said to be Einstein if there exists a constant $c$ such that
\[
R_{\Omega}=c\,g_{\Omega}.
\]

%Note that $K_\Omega(z,z)>0$ away from a proper complex analytic subvariety of $\Omega$ defined by ${\varphi_j=0,\ j=1,\dots}$. 
Since $\Omega\sm \Omega^*$ is contained in a possibly empty complex analytic hypersurface, 
we easily see that the Bergman metric is Einstein(or CSC) in $\Omega^*$ if and only if it is Einstein(or CSC) in a certain open subset of $\Omega^*$. {\it In what follows, we  say that $\Omega$ is Bergman--Einstein if its Bergman metric is Einstein in $\Omega^*$.} (See \cite{HJL25}.)

We have  the following formula:
\begin{equation*}
\begin{split}
    &K_\Omega(z, z)=\sup\{|f(z)|^2:f\in A^2(\Omega),\|f\|=1\}.
\end{split}
\end{equation*}
Further define the following  extremal domain functions (see, e.g, \cite{KYu96} \cite{James}):
\begin{equation*}
\begin{split}
    &\lambda_{\Omega}^k(z):=\sup\Bigl\{\Bigl|\frac{\partial f}{\partial z_k}(z)\Bigr|:\|f\|_{\Omega}=1,\,f(z)=0,\,\frac{\partial f}{\partial z_j}(z)=0\ (0\le j<k)\Bigr\}
\end{split}
\end{equation*}
and $$\lambda_{\Omega}(z)=\lambda^0_\Omega(z)\cdots\lambda^n_\Omega(z).$$
Both the functions $K_{\Omega}(z,  z)$ and $\lambda_{\Omega}(z)$ are monotone decreasing with respect to $\Omega$ (see \cite[Prop 2.2]{KYu96}).

The following formula  relates  the above quantities to the Bergman canonical invariant   $J_{\Omega}$.

\begin{proposition}\cite{KYu96, James}\label{Prop 1-3-31}
Let $\Omega$ be a domain in $\mathbb C^{n+1}$ with $K_{\Omega}(z, z)>0$.  Then
\begin{equation*}
J_{\Omega}(z)=\dfrac{\lambda_{\Omega}(z)}{K^{n+1}_{\Omega}(z, z)}.
\end{equation*}
\end{proposition}

\begin{lemma}\label{26-1-4-lem1}
Let $\Omega$ be a pseudoconvex domain in $\mathbb{C}^{n+1}$ and let $p \in \partial\Omega$ be an $h$--extendible boundary point. Then the Bergman metric $g_\Omega$ is Kähler-Einstein if and only if the Bergman canonical invariant $J_{\Omega}$ satisfies $J_{\Omega} \equiv c_{n+1} := (n+2)^{n+1}\frac{\pi^{n+1}}{(n+1)!}$ on $\widehat\Omega$, where $\widehat\Omega= \Omega \setminus E$ and $E = \{z \in \Omega : K_\Omega(z, z) = 0\}$.
\end{lemma}
\begin{proof}
    Since $p\in\partial\Omega$ is a smooth D'Angelo finite type point, there exist strongly pseudoconvex boundary points arbitrarily close to $p$. Proposition 3.8 and Remark 3.9 in \cite{HJL25} then yield the conclusion of the lemma.
\end{proof}
A main result of this section is the following theorem, whose proof  is based on  the  Bergman maximum domain functions, localization of Bergman canonical invariant functions and  the Boas--Straube--Yu \cite{BSY95} rescaling method.
\begin{theorem}\label{thm1-2-7}
    Let $\Omega\subset\mathbb C^{n+1}$ ($n\geq 1$) be a   pseudoconvex domain which is h-extendible at a boundary point $p$ and let $D$ be its associated local model at $p$. Assume that the Bergman metric of $\Omega$ is Einstein. Then the Bergman metric of $D$ is Einstein.
\end{theorem}
\begin{proof}
Recall that near $p$, there exists a coordinate neighborhood $(U, \varphi)$ of $p$, with $\varphi = (z_0, z_1, \dots, z_n)$ and $\varphi(p)=0$, such that the local defining function for $\Omega \cap U$, which is assumed to be connected,  takes the form
\[
r(z_0, z') = \operatorname{Re} z_0 + P(z') + O\bigl(\sigma(z')^{1+\alpha}\bigr) + O\bigl(|\operatorname{Im} z_0|^{2}\bigr),
\]
and $\Omega \cap U = \{ q \in U : r(\varphi(q)) < 0 \}$, where $0<\alpha<1$ is a certain  constant.

In what follows, we  replace $\Omega \cap U$ by its coordinate representation $\varphi(\Omega \cap U) $, and accordingly treat the function $r$ as being defined on an open subset of $\mathbb{C}^{n+1}$ containing the origin. That is, $p=0$, $U$ is a neighborhood of $0$ and $\Omega\cap U$ is given by 
$$\Omega\cap U:=\{z\in U: r(z_0, z')<0\}$$
with $r(z_0, z') = \operatorname{Re} z_0 + P(z') + O\bigl(\sigma(z')^{1+\alpha}\bigr) + O\bigl(|(\operatorname{Im} z_0)|^{2}\bigr)$.
Let $a(z')$ be the bumping function for $P(z')$ and suppose $0<\delta<1$. Put 
$$\rho_\delta={\rm Re}~(z_0+kz_0^2)+P(z')-\delta a(z').$$
It was proved in \cite{BSY95} that there is a value $k$, independent of $\delta$ such that for each $\delta$, there is a neighborhood $U_{\delta}$ of the origin in 
$\mathbb C^{n+1}$ for which 
\begin{equation*}
    \Omega\cap U_\delta\subset\{z\in U_{\delta}: \rho_\delta(z)<0\}, \quad\forall~ 0<\delta<1.
\end{equation*}
Thus, after the local change of variables $(z_0, z')\mapsto (w_0, w'):=(z_0+kz_0^2, z')$, we make the following assumption: $\Omega$ has a local model $D:=\{z: {\rm Re}~w_0+P(w')<0\}$ at $p$, and for each $0<\delta<1$, there is a neighborhood $U_{\delta}$ of the origin such that  
\begin{equation}\label{26-1-4-b5}
\Omega\cap U_{\delta}\subset D_{\delta}:=\{w\in \mathbb C^{n+1}: {\rm Re }~w_0+P(w')-\delta a(w')<0\}.
\end{equation}
We always assume that $\Omega\cap U_{\delta}$ is connected. 
It follows from the localization of Bergman kernel and extremal domain functions, we have the localization of Bergman canonical invariant function (see \cite[Corollary 3.5]{HJL25}, also \cite[Prop 2.3 and Prop 2.4]{KYu96})
\begin{equation*}
    \lim_{q\rightarrow p}\frac{J_{\Omega}(q)}{J_{\Omega\cap U_\delta}(q)}=1.
\end{equation*}
By Lemma \ref{26-1-4-lem1} and the biholomorphic invariant property of $J_{(\cdot)}$, we have 
\begin{equation}\label{26-1-4-a5}
\lim_{w\rightarrow 0}J_{\Omega\cap U_{\delta}}(w)=c_{n+1}.
\end{equation}

For $t>0$, we consider the scaling map:
\begin{equation*}
    L_t(w):=(t^{-1}w_0, t^{-\frac{1}{m_1}}w_1, \cdots, t^{-\frac{1}{m_n}}w_n), w\in\mathbb C^{n+1}.
\end{equation*}
Put $p_0=(\alpha_0, \alpha_1, \cdots, \alpha_n)$, where $\alpha_0\approx -1$, $\alpha_j\approx 0, 1\leq j\leq n$. Then $p_0\in D$. Write 
$$w_t=(t\alpha_0, t^{\frac{1}{m_1}}\alpha_1, \cdots, t^{\frac{1}{m_n}}\alpha_n).$$
Then for a 
fixed $\delta$ one has $w_t\in \Omega\cap U_{\delta}$ when $t$ is small, and $w_t\rightarrow 0$ as $t\rightarrow 0^{+} $. From (\ref{26-1-4-a5}), we have
\begin{equation}\label{26-1-4-b1}
    \lim_{t\rightarrow 0^{+}}J_{\Omega\cap U_\delta}(w_t)=c_{n+1}.
\end{equation}

Set $\Omega^t_{\delta}:=L_t(\Omega\cap U_{\delta})$. Since $J_{(\cdot)}$ is a biholomorhically invariant and $p_0=L_t(w_t)$, it follows that 
\begin{equation*}\label{26-1-4-a10}
    J_{\Omega\cap U_{\delta}}(w_t)=J_{\Omega^t_\delta}(p_0)=\frac{\lambda_{\Omega^t_\delta}(p_0)}{K_{\Omega^t_\delta}(p_0,   p_0)}
\end{equation*}
and
\begin{equation*}
    \lim_{t\rightarrow0^{+}}\frac{\lambda_{\Omega_\delta^t}(p_0)}{K_{\Omega_\delta^t}(p_0, p_0)}=c_{n+1}.
\end{equation*}
First, since $K_{\Omega}$ is monotone decreasing with respect to $\Omega$ we have 
\begin{equation}\label{26-1-4-b2}
    K_{\Omega^t_\delta}(p_0,  p_0)\leq K_{\Omega^t_\delta\cap D}(p_0,  p_0).
\end{equation}
Second, since $\Omega^t_\delta\cap D\subset D$ and $\Omega^t_\delta\cap D$ converges to $D$ from the interior  as $t\rightarrow 0^+$ in the sense that for any compact subset $K\Subset D$ one has $K\subset \Omega^t_\delta\cap D$ when $t$ is sufficiently small.  By  the Ramadanov theorem \cite{James} we have
\begin{equation}\label{26-1-4-b3}
    \lim_{t\rightarrow 0^+}K_{\Omega^t_\delta\cap D}(p_0,  p_0)=K_{D}(p_0,  p_0).
\end{equation}
along  a sequnece of $t(\rightarrow 0+)$. For simplicity of notation, let us just assume the convergence is for all $t\rightarrow 0+$.
Thus, from (\ref{26-1-4-b2}) and (\ref{26-1-4-b3}) we have
\begin{equation}\label{26-1-4-a3}
    \limsup_{t\rightarrow 0^+} K_{\Omega^t_\delta}(p_0,  p_0)\leq K_{D}(p_0,  p_0).
\end{equation}

On the other hand, since $L_t(D_\delta)=D_\delta$, then from (\ref{26-1-4-b5}) we have for $t>0$,
\begin{equation*}
    \Omega^t_{\delta}\subset D_\delta.
\end{equation*}
It follows that 
\begin{equation}\label{26-1-4-a1}
    K_{\Omega^t_\delta}(p_0,  p_0)\geq K_{D_\delta}(p_0,  p_0)
\end{equation}
From the Lemma in \cite[page 453]{BSY95}, we have
\begin{equation*}
    K_{D_\delta}(p_0, p_0)\rightarrow K_{D}(p_0,  p_0),~\delta\rightarrow 0,
\end{equation*}
uniformly when $p_0\approx(-1, 0, \cdots, 0)$.
Thus, for any $\varepsilon>0$, there exists a $\delta_0$ such that for all $0<\delta<\delta_0$ one has
\begin{equation}\label{26-1-4-a2}
    K_{D}(p_0,  p_0)-\varepsilon<K_{D_\delta}(p_0, p_0)\leq K_D(p_0,  p_0),~ p_0\approx(-1, 0, \cdots, 0).
\end{equation}
Thus, it follows from (\ref{26-1-4-a1}) and (\ref{26-1-4-a2}) that for any $\varepsilon>0$ we can fix some $\delta<\delta_0$ such that 
\begin{equation*}
    K_{\Omega^t_\delta}(p_0,  p_0)\geq K_D(p_0,  p_0)-\varepsilon, \forall t.
\end{equation*}
Taking limit as $t\rightarrow 0^{+}$, we have
\begin{equation*}
    \liminf_{t\rightarrow 0^+}K_{\Omega^t_\delta}(p_0,  p_0)\geq K_D(p_0, p_0)-\varepsilon.
\end{equation*}
Next, taking limits as $\delta\rightarrow 0^+$, we have
\begin{equation*}
    \liminf_{\delta\rightarrow 0^+} \liminf_{t\rightarrow 0^+}K_{\Omega^t_\delta}(p_0, p_0)\geq K_D(p_0, p_0).
\end{equation*}
Combining with (\ref{26-1-4-a3}), we have
\begin{equation}\label{26-1-4-a8}
    \liminf_{\delta\rightarrow 0^+}\liminf_{t\rightarrow 0^+}K_{\Omega^t_\delta}(p_0, p_0)=\limsup_{\delta\rightarrow 0^+}\limsup_{t\rightarrow 0^+}K_{\Omega^t_\delta}(p_0,  p_0)=K_D(p_0,  p_0)
\end{equation}

For  $\lambda_{D}$ , drawing from the results in \cite{KYu96} and \cite{BSY95} with minor adaptation, we also obtain the following  properties. (As noted earlier, for notational ease, we assume that the limit is along the full path as $t\rightarrow 0+$.)

\begin{equation}\label{prop:lambda}
\begin{aligned}
&\lambda_{D} \text{ is monotone decreasing with respect to } D, \\
&\lim_{\delta\rightarrow 0}\lambda_{D_{\delta}}(w)=\lambda_D(w) \text{ uniformly on compact subsets of } D, \\
&\lim_{t\rightarrow 0}\lambda_{\Omega^t_{\delta}\cap D}(p_0)=\lambda_D(p_0) \text{ uniformly for } p_0\approx(-1, 0, \cdots, 0).
\end{aligned}
\end{equation}
 Thus, by a similar argument as above we have
\begin{equation}\label{26-1-4-a9}
    \liminf_{\delta\rightarrow 0^+}\liminf_{t\rightarrow 0^+}\lambda_{\Omega^t_\delta}(p_0)=\limsup_{\delta\rightarrow 0^+}\limsup_{t\rightarrow 0^+}\lambda_{\Omega^t_\delta}(p_0)=\lambda_D(p_0)
\end{equation}
We claim that we can  choose two sequences $\{\delta_j\}, \{t_k\}$ such that
\begin{equation}\label{26-1-4-b6}
    \lim_{j\rightarrow\infty}\lim_{k\rightarrow\infty}K_{\Omega^{t_k}_{\delta_j}}(p_0,  p_0)=K_D(p_0,  p_0),\quad \lim_{j\rightarrow\infty}\lim_{k\rightarrow\infty}\lambda_{\Omega^{t_k}_{\delta_j}}(p_0)=\lambda_D(p_0).
\end{equation}
Indeed, By (\ref{26-1-4-a8}), for all $\varepsilon>0$, there exists $\eta>0$ such that when $0<\delta<\eta$ we have
\begin{equation*}
    K_D(p_0, p_0)-\varepsilon<\liminf_{t\rightarrow 0^{+}} K_{\Omega_\delta^t}(p_0, p_0)\leq \limsup_{t\rightarrow 0^{+}} K_{\Omega_\delta^t}(p_0, p_0)<K_D(p_0, p_0)+\varepsilon.
\end{equation*}
For each $j$,  we replace $\varepsilon$ by $\frac1j$, there exists a $\delta_j\rightarrow 0^{+}$ such that 
\begin{equation}\label{2-27-a1}
\begin{split}
    &\left|\liminf_{t\rightarrow 0^{+}}K_{\Omega_{\delta_j}^t}(p_0, p_0)-K_D(p_0, p_0)\right|<\frac 1j,\\
    &\left|\limsup_{t\rightarrow 0^{+}}K_{\Omega_{\delta_j}^t}(p_0, p_0)-K_D(p_0, p_0)\right|<\frac 1j.
    \end{split}
\end{equation}
When $j=1$, by (\ref{2-27-a1}), $K_{\Omega_{\delta_1}^t}(p_0, p_0)$ is bounded as $t\rightarrow 0$. There exists a sequence $\{t_{1, m}\}_{m=1}^\infty$ such that $K_{\Omega_{\delta_1}^{t_{1, m}}}(p_0, p_0) \rightarrow A_1$ as $m\rightarrow \infty$ and $|A_1- K_D(p_0, p_0)|<1$. Then we take $j=2$. Since by (\ref{2-27-a1}) $K_{\Omega_{\delta_2}^t}(p_0, p_0)$ is bounded as $t\rightarrow 0^{+}$, there exists a subsequence $\{t_{2, m}\}\subset\{t_{1, m}\}$ such that $K_{\Omega_{\delta_2}^{t_{2, m}}}(p_0, p_0)\rightarrow A_2, m\rightarrow\infty$ and $|A_2-K_D(p_0, p_0)|<\frac12$. Thus by an induction,  we have a sequence of subsequences $$\{t_{1, m}\}\supset\{t_{2, m}\}\supset\{t_{3, m}\}\supset\cdots$$
such that
\begin{equation*}
\begin{split}
    &K_{\Omega_{\delta_j}^{t_{j, m}}}(p_0, p_0)\rightarrow A_j,~\text{as}~m\rightarrow \infty~\text{and}\\
    &|A_j-K_D(p_0, p_0)|<\frac{1}{j}.
  \end{split}  
\end{equation*}
Set $t_k=t_{k, k}$. Since $\{t_k\}_{k\geq j}\subset\{t_{j, m}\}_{m=1}^\infty$, we have $\lim_{k\rightarrow\infty}K_{\Omega_{\delta_j}^{t_k}}(p_0, p_0)=A_j$. It follows that 
\begin{equation*}
    \lim_{j\rightarrow\infty}\lim_{k\rightarrow\infty}K_{\Omega_{\delta_j}^{t_k}}(p_0, p_0)=\lim_{j\rightarrow\infty}A_j=K_{D}(p_0, p_0).
\end{equation*}
By a similar argument, we can assume that 
\begin{equation*}
    \lim_{j\rightarrow\infty}\lim_{k\rightarrow\infty}\lambda_{\Omega_{\delta_j}^{t_k}}(p_0)=\lambda_{D}(p_0).
\end{equation*}
Thus,  
\begin{equation*}
    J_D(p_0)=\frac{\lambda_D(p_0)}{K_D(p_0, \overline p_0)}=\lim_{j\rightarrow\infty}\lim_{k\rightarrow\infty}\frac{\lambda_{\Omega_{\delta_j}^{t_k}}(p_0)}{K_{\Omega_{\delta_j}^{t_k}}(p_0, p_0)}=\lim_{j\rightarrow\infty}\lim_{k\rightarrow\infty}J_{\Omega_{\delta_j}^{t_k}}(p_0)=\lim_{j\rightarrow \infty} c_{n+1}=c_{n+1}.
\end{equation*}
when $p_0\approx(-1, 0, \cdots, 0)$.
By the real analyticity  of $J_D$ and the fact that $D$ is  connected, one has $$J_D(z)\equiv c_{n+1}, \forall z\in D.$$
Thus, the Bergman metric of $D$ is K\"ahler--Einstein and we complete the proof  of Theorem \ref{thm1-2-7}.
\end{proof}
For the constant scalar curvature case, we replace $\lambda_{\Omega}$ with $N_{\Omega}$ as defined in \cite{KYu96}, where $N_\Omega(z)$ is given by
\begin{equation*}
\begin{aligned}
N_\Omega(z)=\sup\Big\{
& K^{2n}_{\Omega}(z, z) \operatorname{tr}\Big( f''(z)\,\overline{\operatorname{ad} G_{\Omega}(z)}\,\overline{f''(z)}\,\operatorname{ad} G_{\Omega}(z) \Big) \\
& : \|f\|_{\Omega}=1,\ f(z)=f'(z)=0
\Big\}.
\end{aligned}
\end{equation*}
Here, $\operatorname{ad} G_{\Omega}(z)$ denotes the adjoint matrix of $G_{\Omega}(z)$, $\operatorname{tr}(\cdot)$ is the matrix trace, and $f''(z)=\left(\frac{\partial^2 f}{\partial z_i\partial z_j}\right)$. It is known from \cite[Prop. 2.2, Prop. 2.5]{KYu96} that the quantity $N_{\Omega}$ satisfies the same properties as $\lambda_{\Omega}$ listed in (\ref{prop:lambda}).

From Proposition 2.1 (iv) in \cite{KYu96}, we have the scalar curvature formula
\[
S_{\Omega}(z)=(n+1)(n+2)-\frac{1}{K_{\Omega}^{2n+3}(z, z)J_{\Omega}(z)}N_{\Omega}(z).
\]

We note that the localization result for \(N_{\Omega}\) in Proposition~2.4 of \cite{KYu96} remains valid for unbounded pseudoconvex domains, as evidenced in \cite{HJL25}. Therefore, the proof of Theorem~\ref{thm1-2-7} carries over verbatim to Bergman metrics of constant scalar curvature, yielding the following analogous theorem.

\renewcommand{\thetheorem}{3.3'}
\begin{theorem}\label{thm1-2-7-CSC}
    Let $\Omega\subset\mathbb C^{n+1}$ ($n\geq 1$) be a pseudoconvex domain that is $h$-extendible at a boundary point $p$, and let $D$ be its associated local model of $\Omega$ at $p$. If the Bergman metric of $\Omega$ has constant scalar curvature, then the Bergman metric of $D$ also has constant scalar curvature.
\end{theorem}
\renewcommand{\thetheorem}{\arabic{section}.\arabic{theorem}}
%-----------------------------------------------
\section{Rationality  of germs of CR maps  into sphere}

Now, suppose that the Bergman metric of an $h$-extendible model domain is Kähler-Einstein. By Theorem \ref{spherical}, all strongly pseudoconvex boundary points are spherical, meaning there exist local biholomorphic maps sending a neighborhood of each such point to a piece of the unit sphere. A result of Mir–Zaitsev \cite{MZ21} (see also \cite{MMZ03, LMR23}) states that such a local map extends to any weakly pseudoconvex boundary point. In this section, we prove that this local map is, in fact, rational. This proof is of independent interest and is included here in detail, even though it goes beyond what is required for the proof of Theorem \ref{Main theorem 2}; readers interested solely in the proof of Theorem \ref{Main theorem 2} may skip this section. Finally, we note that a more general rationality theorem for smooth CR maps into spheres from non-degenerate quadrics was established by Forstnerič in \cite{Fo89, Fo92}.

Let $P(z,\overline{z})=O(|z|^2)$ be a real valued  polynomial. We  define the domain $D$ as follows:
$$D:=\{(z, w)\in\mathbb C^{n}\times\mathbb C: w+\overline w>P(z, \overline z)\}.$$
We denote by $M$ the boundary of $D$. Assume further that $P$ is plurisubharmonic and $D$ is of finte D'Angelo type. Namely, $M$ contains no non-trivial holomorphic curves.

Note that when $D$ is an  $h$--extendible model,   $P(z, \overline z)$ has no harmonic terms and is a weighted homogeneous polynomial.  Since $D$ is of finite D'Angelo type at $0$ and any boundary point can be  mapped to a boundary point arbitrarily  close  to $0$, by the D'Angelo stability theorem \cite{D82}, $D$ is of finite D'Angelo type at any boundary point.

For $(z, w)\in\mathbb C^{n}\times\mathbb C$, the Segre variety of $D$ at $(z, w)$ denoted by $Q_{(z, w)}$ is given by 
$$Q_{(z, w)}:=\{(\widehat z, \widehat w)\in\mathbb C^{n}\times\mathbb C: \widehat w+\overline w=P(\widehat z, \overline z)\}.$$

First, we fix a large $R>0$. We next proceed  to define a reflection map $\mathcal R$.
\begin{lemma}
For any sufficiently large $R$, there exist $R\ll R'\ll R''$ such that for  $v\in\mathbb C^{n}$ satisfing $\|v\|<\varepsilon(R)\ll 1$, where $\varepsilon(R)$ is a small positive constant depending only on $R$ and  for any $(z, w)\in B_R(0):=\{(z, w)\in\mathbb C^{n}\times\mathbb C: \|(z, w)\|<R\}$, the complex line 
$$\mathcal L_v:=\{(z+tv, w+t): t\in\mathbb C\},$$
that passes $(z,w)$ in the direction $(v,1)$,
intersects $Q_{(z, w)}$  exactly at one point in $B_{R'}(0)$. 
Moreover, $\mathcal L_v$ intersects $Q_{(z,w)}$ for $(z,w)\in B_{R'}(0)$  exactly one point in $B_{R''}(0)$.
\end{lemma}
\begin{proof}
Since 
\begin{equation*}
Q_{(z, w)}\cap\mathcal L_v=\{(\widehat z, \widehat w):\widehat z=z+v(\widehat w-w), ~\widehat w=-\overline w+P(z+v(\widehat w-w),\overline z)\}.
\end{equation*}
When $v=0$, there is a unique solution $\widehat w=-\overline w+P(z, \overline z)$ for all $R$, which uniquely determines $\widehat z=z$. Since 
\begin{equation*}
\frac{\partial P(z+v(\widehat w-w), ~\overline z)}{\partial\widehat w}\rightarrow 0 ~\text{as}~v\rightarrow 0,
\end{equation*}
for a fixed $R\gg1$,  when $|v|\ll1$ we can find $R''\gg R'\gg R$ such that the intersection admits a unique solution $(\widehat z, \widehat w)\in B_{R'}(0)$ for any $(z, w)\in B_R(0)$ and a unique intersection for $(\widehat z, \widehat w)\in B_{R''}(0)$ for any $(z, w)\in B_{R'}(0)$.
\end{proof}
Now for any $(z, w)\in B_{R'}(0)$, we define the reflection map by
\begin{equation*}
\mathcal R_{R, v}(z, w) := \mathcal L_v\cap Q_{(z, w)}\cap B_{R''}(0).
\end{equation*}
Then $(\widehat z, \widehat w): = \mathcal R_{R, v}(z, w)$ is determined by
\begin{equation*}
\begin{split}
&\widehat w = -\overline w + P\bigl(z+v(\widehat w-w),\,\overline z\bigr),\\
&\widehat z = z+v(\widehat w-w),
\end{split}
\end{equation*}
with $(\widehat z, \widehat w)\in B_{R''}(0)$.  
Since for any $(z, w)\in \partial D$, we have $(z, w)\in Q_{(z, w)}$, it follows that
$$\mathcal R_{R, v}|_{\partial D\cap B_R(0)} = \mathrm{Id}.$$
Moreover, the relation that $(\widehat z, \widehat w)\in Q_{(z, w)}$ is equivalent to $(z, w)\in Q_{(\widehat z, \widehat w)}$; consequently,
$$\mathcal R^2_{R, v} (z,w)= (z,w)\ \hbox{for } (z,w)\in B_R(0). $$
\begin{lemma}\label{2-3-lem1}
Let $S_0\subset B_R(0)$ be a smooth germ of a complex analytic  hypervariety $S$ in $B_{R''}(0)$.
    For an open dense subset of $\{v\in\mathbb C ^{n}: \|v\|<\varepsilon(R)\}$, we have that $\mathcal R_{R, v}(S_0)$ is not a germ of $S$.
\end{lemma}
\begin{proof}
Let \(p_0 \in S_0\). First, suppose that near \(p_0\), \(S_0\) is the graph of a holomorphic function \(w = h(z)\), or equivalently, that the vector \(\frac{\partial}{\partial w}\big|_{p_0}\) is transverse to \(T_{p_0}^{1,0}S_0\). Taking \(v = 0\), the set \(\mathcal R_{R,0}(S_0)\) near \(\mathcal R_{R, 0}(p_0)\) is described by
\[
w = -\overline{h(z)} + P(z,\overline{z}),
\]
where 
\[
\mathcal R_{R,0}(z,w) = \bigl(z,\; -\overline{h(z)} + P(z,\overline{z})\bigr), ~(z, w)\approx p_0 ~\text{and}~ (z, w)\in S_0.
\]
Since \(-\overline{h(z)} + P(z,\overline{z})\) is not holomorphic in \(z\), the image \(\mathcal R_{R,0}(S_0)\) is not a complex submanifold near \(\mathcal R_{R,0}(p_0)\). Consequently, \(\mathcal R_{R,0}(S_0)\) cannot coincide with \(S\) in any neighbourhood of \(\mathcal R_{R,0}(p_0)\).

Now, suppose that for every point \(p\) near \(p_0\) we have \(\frac{\partial}{\partial w}\big|_{p} \in T_p^{1,0}S_0\). 
We may assume that near \(p_0\) the set \(S_0\) is described by
\begin{equation*}
z_1 = h(z',w),\qquad z' = (z_2,\dots ,z_{n}).
\end{equation*}
Since \(\frac{\partial}{\partial w}\) is tangent to \(S_0\) near \(p_0\), we have
\begin{equation*}
\frac{\partial}{\partial w}\bigl(-z_1+h(z',w)\bigr)\equiv 0,
\end{equation*}
and therefore
\begin{equation*}
h(z',w)=h(z')\quad\text{is independent of } w.
\end{equation*}
Thus, near \(p_0\) the set \(S_0\) can be written as
\begin{equation*}
z_1 = h(z'),\quad z' = (z_2,\dots ,z_n).
\end{equation*}
Consequently,
\begin{equation*}
\mathcal R_{R,v}(S_0)=\{(\widehat z,\widehat w):\widehat z=z+v(\widehat w-w),\;\widehat w=-\overline w+P(z+v(\widehat w-w),\overline z),\;z_1=h(z')\}.
\end{equation*}

Suppose further that, as germs, we have \(\mathcal R_{R,v}(S_0)\subset S\). Then $\mathcal R_{R, v}(S_0)$ is a germ of complex analytic hypervariety and we assume that $\mathcal R_{R, v}(S_0)=S$ near $\mathcal R_{R, v}(p_0)$.  Recall that \(\mathcal R_{R,v}(S_0)\) is given by the system
\begin{equation*}
\begin{split}
&\widehat z = z+v(\widehat w-w),\\
&\widehat w = -\overline w + P\bigl(z+v(\widehat w-w),\overline z\bigr),\\
&z_1 = h(z').
\end{split}
\end{equation*}
Choosing \(v_2=\cdots =v_{n}=0\), we obtain that 
\(\mathcal R_{R,v}(S_0)\) is described by
\begin{equation*}
\begin{split}
&\widehat z_1 = h(\widehat z_2,\dots ,\widehat z_{n}) + v_1(\widehat w-w),\\
&\widehat w = -\overline w + P\bigl(z+v(\widehat w-w),\overline z\bigr).
\end{split}
\end{equation*}
Then
\begin{equation*}
    v_1(\widehat w-w)=\widehat z_1-h(\widehat z_2, \cdots, \widehat z_n),\quad (\widehat z_1, \cdots, \widehat z_n, \widehat w)\in S,
\end{equation*}
and $v_1(\widehat w-w)$ as a function of $(\widehat{z},\widehat{w})\in S$ is independent of the parameter $v_1$. 
Here, we regard $w$ as a real analytic function in $(\widehat{z},\widehat{w})\in S$ with parameter $v_1$. Since 
$$\overline w = -\widehat w + P\bigl(\widehat{z}, \overline{\widehat{z_1}-v_1(\widehat{w}-w)},
\overline{\widehat{z_2}},\cdots,\overline{\widehat{z_n}}), \quad (\widehat z_1, \cdots, \widehat z_n, \widehat w)\in S,
$$
$w$ as a function on $S$ with parameter $v_1$ is independent of $v_1$, too.
Hence, it follows that over $S$
\[
\widehat w - w \equiv 0,
\]
and consequently on $S_0$
\[
w+\overline w = P(z,\overline z).
\]
Thus, for \((z,w)\) near \(p_0\) with \((z,w)\in S_0\), we have \((z,w)\in  M\). This implies \(S_0 \subset  M\), which contradicts the assumption  that \(M\) is of finite type.

Therefore, we have proved: If \(S\) is a complex analytic variety of codimension one in \(\mathbb C^{n+1}\) and \(S\cap B_R(0)\neq \emptyset\), then there exists a vector \(v\in\mathbb C^{n}\) with \(\|v\|\ll 1\) such that the intersection \(\mathcal R_{R,v}(S\cap B_R(0))\cap S\) has real codimension at least two in  $S$.
\end{proof}

We now prove the following rationality theorem for the map:
\begin{theorem}\label{thm1-2-26}
    Suppose $F$ is a CR diffeomorphism from an open piece of $\partial D$ into $\partial \mathbb B^{n+1}$. Then $F$ extends to a rational holomorphic  map, whose   poles  are outside of $\overline D$,  $F:\overline  D\rightarrow \overline {\mathbb B^{n+1}}$ with $F(\partial D)\subset\partial \mathbb B^{n+1}$. 
\end{theorem}
\begin{proof}
We divide the proof of the theorem into several steps.

 \textit{Step 1. We first show that \(F\) is a rational map.} 

 By the Webster algebraicity theorem \cite{We77, Hu94},  \(F\) is algebraic. Suppose, for a contradiction, that  \(F\) is not rational. Then its branch locus denoted by $S$ is a complex analytic variety of codimension one in $B_R(0)$ for $R\gg1$. Choose a strongly pseudoconvex point $p_0\in\partial D$. There exists a loop \(\gamma\) in \(B_R(0)\setminus S\)   with \(\gamma(0)=\gamma(1)=p_0\) such that the analytic continuation of  \(F\)  along  \(\gamma\) yields a new holomorphic branch 
  \(F_2\) satisfying \(F_2(p_0) \neq F_1(p_0)\), where \(F_1=F\).

Applying the Thom transversality theorem and a homotopic perturbation, we may assume \(\gamma\) has the factorization
\[
\gamma = \gamma_1^{-1} \circ \gamma_2 \circ \gamma_1,
\]
with \(\gamma\bigl([0,1]\bigr) \subset B_R(0) \setminus S\). Here, \(\gamma_1\) is a simple curve joining \(p_0\) to a point \(p^*\in B_R(0)\sm S\), \(\gamma_1^{-1}\) is its reverse, and \(\gamma_2\) is the positively oriented boundary of a small closed  holomorphic disk, which in a local chart, can be expressed as
\[
D_{\varepsilon_0} = \{(0,\dots,0,z_{n+1}): |z_{n+1}| \le  \varepsilon_0\},
\]
and intersects $S$  only at a certain  smooth point $p^\sharp\in S$, where \(\varepsilon_0 \ll 1\).

Without loss of generality, we may take \(p^\sharp = 0\), so that near \(p^\sharp\) the variety \(S\) is defined by \(z_{n+1} = 0\) in this local chart and the loop \(\gamma_2\) is the counterclockwise-oriented circle
\[
\gamma_2 = \{(0,\dots,0,z_{n+1}): |z_{n+1}| = \varepsilon_0\}.
\]

Now we define a holomorphic map $\mathcal F$ from an open piece of $$\mathcal M_{\partial D}:=\{((z, w),~(\xi, \eta))\in\mathbb C^{n+1}\times\mathbb C^{n+1}: w+\eta=P(z, \xi)\}$$
    to an open piece of 
    $$\mathcal M_{\partial\mathbb B^{n+1}}:=\{((\widehat z, \widehat w),~ (\widehat\xi, \widehat\eta))\in\mathbb C^{n+1}\times\mathbb C^{n+1}: \widehat w\cdot\widehat \eta+\widehat z\cdot\widehat \xi=1\}$$
    with
    \begin{equation*}
        \mathcal F=(F(z, w), \overline F(\xi, \eta)), ((z, w), ~(\xi, \eta))\approx(p_0, \overline{p_0}),
    \end{equation*}
    where $\overline F(\xi, \eta):=\overline{F(\overline\xi, \overline\eta)}$.
    Then \(\mathcal F\) sends a neighborhood of \((p_0, \overline {p_0})\) in \(\mathcal M_{\partial D}\) into a neighborhood in \(\mathcal M_{\partial \mathbb B^{n+1}}\). By the definition of \(\mathcal R_{R, v}\), the curve
    \((\mathcal R_{R, v}(\gamma),\overline{\gamma})\)
  lies in \(\mathcal M_{\partial D}\).

We  analytically continue \(\mathcal F\) along \((\mathcal R_{R, v}(\gamma), \overline\gamma)\).
By Lemma \ref{2-3-lem1} we may assume \(\mathcal R_{R, v}(p^\sharp)\notin S\). After perturbing \(\gamma\) and shrinking \(\varepsilon_0\ll1\) if necessary, we can also ensure that \(\mathcal R_{R, v}(\gamma) \subset \mathbb C^{n+1}\setminus S\) and that \(\mathcal R_{R, v}(\gamma)\) is null‑homotopic in \(\mathbb C^{n+1}\setminus S\). Consequently,
\[
F_1\bigl(Q_{(z, w)}\cap U\bigr) \subset Q_{F_2(z, w)},\qquad (z,w)\approx p_0,
\]
where \(U\) is a small neighbourhood of \(p_0\).

Observe that \(F_1(Q_{(z, w)}\cap U) \subset Q_{F_1(z, w)}\) for $(z,w)\approx p_0$. For the complex unit ball \(\mathbb B^{n+1}\) the correspondence \(Z \leftrightarrow Q_Z\) is bijective; therefore
\[
Q_{F_2(z, w)} = Q_{F_1(z, w)}
\]
as both contains  the same open piece  \(F_1(Q_{(z, w)}\cap U)\).
We thus conclude that  \(F_1(p_0)=F_2(p_0)\), which contradicts the choice of the loop \(\gamma\). Hence \(F\) is a rational map.

\textit{Step 2. We show the poles of $F$ are outside of $\overline D$.}
We write \(F\) in the form  
\[
F = \frac{(p_1,\dots ,p_{n+1})}{q},
\]
where \(p_j\) and \(q\) are polynomials with
$$\gcd(p_1,\dots ,p_{n+1},q)=1, \quad 1\leq j\leq n+1.$$
The set \(E:=\{q=0\}\) is called the {pole divisor} of \(F\).
Now we apply a result of Chiappari \cite{Ch91} and conclude that $E\cap M=\emptyset$ and $F(M)\subset \partial{\mathbb B}^{n+1}$. Here we mention that although Chiappari \cite{Ch91} stated his theorem under the additional assumption that $F$  is holomorphic on one side of $M$ , the proof in fact goes through without this extra hypothesis.

 We further claim that no irreducible component \(E_1\) of the pole divisor \(E\) is contained in ${D}$. Indeed, if not, there is some \(a>0\)  so that  
\[
E_1\subset\overline{D_a}\quad\text{and}\quad E_1\cap\partial D_a\neq\emptyset,
\]
where  
\[
D_a:=\{(z,w): w+\overline w>P(z, \overline z)+a\}.
\]
Applying the maximum principle on the complex variety \(E_1\) to the plurisubharmonic function \(-w-\overline w+P(z, \overline z)+a\) forces \(E_1\subset\partial D_a\). However, \(\partial D_a\) is also of finite type in the sense of D’Angelo, which yields a contradiction. Consequently, \(F\) is holomorphic on \(\ov{D}\). Again, applying  the maximum principle to $\|F\|^2-1$ in $D$, we see that
$F(D)\subset {\mathbb B}^{n+1}$.
\end{proof}

\section{ Existence of non-stongly pseudconvex $h$--extendible  points on  a real analytic hypersurface of finite type}

In this section, we establish a key result on the existence of weakly pseudoconvex $h$-extendible boundary points. This result allows us to reduce the proof of Theorem \ref{Main theorem 3}, and more generally Theorem \ref{Main theorem 3-1}, to Theorem \ref{Main theorem 2}.

\begin{theorem}\label{thm1-2-23} Let $M\subset\mathbb C^{n+1}$ be a real analytic pseudoconvex hypersurface of finite D'Angelo type. Let $p_0\in M$ be a non-strongly pseudoconvex  point and let $U$ be a neighborhood of $p_0$ in $\mathbb C^{n+1}$ and write $\Omega$ for a pseudoconvex side of $M$ in $U$. Assume that $F$ is a holomorphic map from $U$ into $\mathbb C^{n+1}$, that is biholomorphic away from a complex analytic variety of codimension one  in $U$, such that $F(U\cap M)\subset \partial {\mathbb B}^{n+1}$. 
Then there exist a weakly pseudoconvex point $p^\ast\in M$ near $p_0$, a positive integer $m>1$, and a  holomorphic coordinate system  $(t_1, \cdots, t_{n+1})$ centered  at $p^\ast$ with $t(p^*)=0$, such that $\Omega$ near $p^\ast$ is defined by 
\begin{equation*}
    {\rm Re} ~t_{n+1}>|t_1|^{2m}+\sum_{j=2}^{n}|t_j|^2+ O((|t_1|^{2m}+\sum_{j=2}^{n}|t_j|^2)^{1+\alpha})+O(|{\rm{Im}}~ t_{n+1}|^{1+\gamma}),
\end{equation*}
where $0<\alpha, \gamma<1$ are constants.
In particular, $p^*$ is a weakly pseudoconvex $h$--extendible boundary point of $\Omega$.
\end{theorem}
\begin{proof}[Proof of Theorem \ref{thm1-2-23}]
Write \( E:=\{z\in U: J_F(z)=0\}\), where \(J_F(z)=\det F'(z)\) and \(F'(z)\) denotes the Jacobian matrix of \(F\) at \(z\). Then \( E\cap M\) coincides precisely with the set of points at which \( M\cap U\) fails to be strongly pseudoconvex. In particular, $p_0\in E$.

We first  prove the following lemma:
\begin{lemma}\label{claim10}$E$ intersects $M$  transversely at a smooth point $p_1$ of $E$ with $p_1$ sufficiently close to $p_0$. Moreover $M\cap E$ is also a smooth hypersurface in $E$ near $p_1$.
\end{lemma}
\begin{proof}[Proof of Lemma \ref{claim10}]
Since  $M$ is pseudoconvex and contains no non-trivial holomorphic curves, by the pseudoconvexity of $M$, for any irreducible component $E_1$ of $E$ with $p_0\in E_1$, $E_1 \not \subset \overline\Omega$ near $p_0$. 

Now assume that  $E$ stays completely outside $U\cap \Omega$ near $p_0$. 
After a holomorphic change of coordinates, we assume that $p_0=0$  and $F(0)=e_1=(0, \cdots, 0, 1)\in\partial \mathbb B^{n+1}.$ Again, by the maximum principle and the  Hopf lemma,   $|F|^2-1$   has a positive   derivative along an outward normal direction at any point on $M\cap U$. After shrinking $U$ if necessary, we thus conclude that 
\begin{equation*}
    F(U\cap \Omega)\subset\mathbb B^{n+1},~F(U\cap M)\subset\partial \mathbb B^{n+1},~\text{and}~F(U\cap\overline \Omega^c)\subset \mathbb C^{n+1}\setminus \overline{\mathbb B^{n+1}}.
\end{equation*}
Since $F^{-1}(\{e_1\})$ is a complex analytic variety and $M$ contains no non-trivial holomorphic curves, after shrinking $U$ if needed, we have $$\sharp\{F^{-1}(\{e_1\})\cap U\}=1.$$ This implies that $F$ is a local proper holomorphic map from a neighborhood of $0$ into a neighborhood of $e_1$ in $\mathbb C^{n+1}$. 
Notice that $0\in E\subset {\Omega}^c$.  We next find a simply connected pseudoconvex side $\Omega_0^\ast$  of  $\partial {\mathbb B}^{n+1}$ near  $e_1$ such that $\Omega_0^\ast \Subset F(U)$. Write $\Omega_0$ for the connected component of $\Omega\setminus F^{-1}(\partial\Omega_0^\ast)$ near $0$. 
$\Omega_0$ contains  a peudoconvex side of $M$ near $0$.  Then $F$ is a proper holomorphic map from $\Omega_0$ to $\Omega_0^\ast$. Since $F$ is a local biholomorphic map, $F$ is a covering map. Thus, $F^{-1}$ is also a biholomorphic from 
 $\Omega_0^*$ to $\Omega_0$  as $\Omega_0^*$ is simply connected.  We claim now $F^{-1}$ is continuous up to $\partial {\mathbb B}^{n+1}$ near $e_1$. Indeed, for any $q(\in \partial {\mathbb B}^{n+1})\approx e_1$, the cluster set of $F^{-1}$ at $q$ is the  finite subset $F^{-1}(\{q\})\subset  M\cap \p\Omega_0$  which is mapped to $q$. Suppose it is not a single point,  we can find two sequences $\{w_j^{(1)}\}, \{w_j^{(2)}\}\subset\Omega_0^\ast$ converging to $q$ such that $z_j^{(1)}=F^{-1}(w_j^{(1)})\rightarrow z^{(1)}\in M$ and $z_j^{(2)}=F^{-1}(w_j^{(2)})\rightarrow z^{(2)}\in M$. Here, $z^{(1)}$ and $z^{(2)}$ have the least positive distance between any two points $F^{-1}(\{q\})$. Now connecting  $w_j^{(1)}$ and $w_j^{(2)}$  by a segment and find a point on this segment  such that  its image  by $F^{-1}$ has the same distance to  $z_j^{(1)}$ and $z_j^{(2)}$. Then we find a point in the cluster set  of $F^{-1}$ at $q$  whose distance  to either $z^{(1)}$ or $z^{(2)}$ is half the original minimum one. This is a contradiction. Hence, $F^{-1}$ is continuous up to the boundary near $e_1$.
 Now, by a result of Bell--Catlin \cite{BC88}, $F^{-1}$ extends smoothly to $\p{\mathbb B}^{n+1}$ near $e_1$. Hence, $F$ is a CR diffeomorphism from $0$ to $e_1$.  Thus, $0$ is a strongly pseudoconvex point. This is a contradiction to our assumption.

Next, let $E_1$ be an irreducible component of $E$ passing through  $p_0$. Also suppose that in any connected small neighborhood  $U^*$ of  $p_0$  in $\CC^{n+1}$ for which  $U^*\cap E_1$ is connected, it contains points of $E_1$ on both sides of $M$.
Since  $(M\cap E_1)\cap U^*$ is a real analytic subvariety in $U^*$, if it has Hausdorff codimension at least  two in $E_1\setminus \hbox{Sing}(E_1)$, then $U^*\cap (E_1\setminus \hbox{Sing}(E_1)) \setminus (M\cap E_1)$ is connected (see Rudin \cite[Chapter 14]{Ru80}). This yields a contradiction, as we could then connect an outside and an interior  point of $\Omega$ by a  continuous real curve in $E_1\setminus \hbox{Sing}(E_1)$  along which a smooth defining function of $\Omega$ takes both positive and negative values but never zero.
 If $(M\cap E_1)\cap U^*$ has Hausdorff codimension one in $E_1$, then a generic point in $(M\cap E_1)\cap U^*$ is  a smooth point of $E$ and $E\cap M$. Then $E\cap M$ is a real hypersurface of finite D'Angelo type  in $E$ near such a smooth point. By the maximum principle and by  the  Hopf lemma applied  to   $|F|^2-1$ restricted to a small embedded holomorphic disk smooth up to the boundary and  attached to $E\cap M$ that also passes  through such a point,  it   has a positive   derivative along an outward normal direction of $\Omega\cap E$ at this  point. Hence, $E$ intersects $M$ transversely  at this  point.

 The proof of the  lemma is complete.
 \end{proof}

Moving to a nearby point and choosing a local holomorphic chart centered at this point if needed, we now assume that $E$ is smooth and $E$ intersects $M$ transversally  at $0$.

Assume, without loss of  generality, that  $M$ is defined near the origin by ${\rm Re} ~z_{n+1}=O(|z|^2)$ and  \(E\) is given by
\[
z_1 = h(z_2, \dots, z_{n+1}), \qquad h(0)=0.
\]
By the Remmet proper mapping theorem, $F(E)$ is a complex analytic hypervariety near $e_1$.
After moving to a nearby point again if needed, we can further  assume that the image \(F( E)\) is also  smooth near \(e_1\), $F^{-1}(F(E))=E$ near $0$ and $F(E)$ is defined  by an equation of the form, say, 
\[
w_2 = g(w_1, w_3, \dots, w_{n+1}), \qquad g(0)=0.
\]
Here, we choose the Heisenberg coordinates, denoted by  $\mathbb H^{n+1}$,  $w=(w_1, \cdots, w_{n+1})$  which sends $e_1$ to $0$, and $\partial \mathbb H^{n+1}$ is given by 
$${\rm Re}~w_{n+1}=\sum_{j=1}^n|w_j|^2.$$
Perform the coordinate changes
\[
\widetilde z_1 = z_1 - h(z_2, \dots, z_{n+1}), \qquad 
\widetilde z_j = z_j \quad (j=2,\dots,n+1),
\]
and
\[
\widetilde w_1 = w_2 - g(w_1, w_3, \dots, w_{n+1}), \quad 
\widetilde w_2 = w_1, \quad 
\widetilde w_j = w_j \quad (j=3,\dots,n+1).
\]
Then, in terms of the tilde coordinates $(\widetilde w_1, \cdots, \widetilde w_{n+1})$ near $0$, $M$ is given by
\begin{equation}\label{tilde form of ball}
{\rm Re}~\widetilde w_{n+1}=|\widetilde w_1+g(\widetilde w_2, \cdots, \widetilde w_{n+1})|^2+\sum_{j=2}^n|\widetilde w_j|^2.
\end{equation}
In the \(\widetilde z\)-coordinates, we have \( E = \{\widetilde z_1 = 0\}\), while in the \(\widetilde w\)-coordinates, \(E^\ast:=F( E)\) is defined by  \(\{\widetilde w_1 = 0\}\).

Let \(\widetilde w=F(\widetilde z)=(\widetilde f_1,\cdots,\widetilde f_{n+1})\) denote the map expressed in these new coordinates. Since $\widetilde f_1=0$ if and only if $\widetilde{z_1}=0$, we can then write 
\[
\begin{aligned}
\widetilde f_1 &= \widetilde z_1^m \, g_1(\widetilde z), \qquad g_1(0)\neq 0, \; m\ge 2,\\[4pt]
\widetilde f_j &= \widetilde a_j(\widetilde z_2, \dots, \widetilde z_{n+1}) + \widetilde z_1 \, \widetilde b_j(\widetilde z), \qquad j=2,\dots, n+1.
\end{aligned}
\] 
Since $F$ is also proper  from  $E$ to $E^\ast$ and $F|_E=(0, \widetilde a_2, \cdots \widetilde a_{n+1})$, moving to a nearby point in $M$  along the direction $(\widetilde z_2, \cdots, \widetilde z_{n+1})$ if needed, we assume that   \((\widetilde a_2, \dots, \widetilde a_{n+1})\) is a biholomorphic map from $(\mathbb C^{n}, 0)$ to $(\mathbb C^{n}, 0)$.

Introduce the further change of variables
\[
\begin{aligned}
\widetilde{\widetilde z}_1 = \widetilde z_1 \, g_1^{1/m}(\widetilde z);\quad
\widetilde{\widetilde z}_j= \widetilde a_j(\widetilde z_2, \dots, \widetilde z_{n+1}) + \widetilde z_1 \, \widetilde b_j(\widetilde z), \quad j=2,\dots, n+1.
\end{aligned}
\]
In the \(\widetilde{\widetilde z}\)-coordinates, we obtain
\[
\widetilde f_1 = \widetilde{\widetilde z}_1^{\,m}, \qquad 
\widetilde f_j = \widetilde{\widetilde z}_j \quad (j=2,\dots,n+1).
\]
For notational simplicity, we again write  \(\widetilde z\) in place of \(\widetilde{\widetilde z}\) and thus,
\begin{equation}\label{normalization map}
 F(\widetilde z) = (\widetilde z_1^{\,m},\; \widetilde z_2,\dots, \widetilde z_{n+1}).
\end{equation}
%Recall the relation between the original and the tilde coordinates,\[\begin{aligned}w_1 = \widetilde w_2,\quadw_2 = \widetilde w_1 + \widetilde h(\widetilde w_2, \widetilde w_3, \dots, \widetilde w_{n+1}),\quadw_j = \widetilde w_j \quad (j\ge 3).\end{aligned}\]
%We then  express the components of \(F\) as\[\begin{aligned}f_1(\widetilde z) = \widetilde z_2,\quadf_2(\widetilde z) = \widetilde z_1^{\,m} + \widetilde h(\widetilde z_2, \widetilde z_3, \dots, \widetilde z_{n+1}),\quadf_j(\widetilde z) = \widetilde z_j \quad (j\ge 3).\end{aligned}\]
It follows from (\ref{tilde form of ball})  and (\ref{normalization map}) that $M$ near $0$  is  defined by 
\begin{equation*}
    {\rm Re}~\widetilde z_{n+1}=|\widetilde z_1^m+g(\widetilde z_2, \widetilde z_3, \cdots, \widetilde z_{n+1})|^2+\sum_{j=2}^{n}|\widetilde z_j|^2.
\end{equation*}
Write $g(\widetilde z_2, \widetilde z_3, \cdots, \widetilde z_{n+1})=a_2\widetilde z_2+\cdots+a_{n+1}\widetilde z_{n+1}+ O(2).$
After a unitary change of coordinates in \((\widetilde z_2, \dots, \widetilde z_{n})\), we may assume that  
\begin{equation*}
    \operatorname{Re}\widetilde z_{n+1} = |\widetilde z_1^m + a_0 \widetilde z_2|^2 + \sum_{j=2}^{n} |\widetilde z_j|^2 + O\bigl(\sigma(\widetilde z')^{1+\alpha}\big)+O\big(|\widetilde z_{n+1}|^{1+\gamma}\bigr),
\end{equation*}
for  certain constants $0<\alpha, \gamma<1$,
where $\widetilde z=(\widetilde z_1, \cdots, \widetilde z_n, \widetilde z_{n+1}):=(\widetilde z',~ \widetilde z_{n+1})$ and
\[
\sigma(\widetilde z') := |\widetilde z_1|^{2m} + \sum_{j=2}^{n} |\widetilde z_j|^2.
\]
Define new coordinates  
\[
\widetilde z_1 = b^{\frac1m}t_1,\quad \widetilde z_2 = t_2 + a t_1^m,\quad \widetilde z_j = t_j\ (j\ge 3)
\]
with constants \(a\in {\mathbb R}, b\not =0\) to be chosen later.  
A direct computation yields
\[
\begin{split}
|\widetilde z_1^m + a_0 \widetilde z_2|^2 + |\widetilde z_2|^2
&= \bigl|(b + a_0 a) t_1^m + a_0 t_2\bigr|^2 + \bigl|a t_1^m + t_2\bigr|^2 \\[2pt]
&= \bigl(|b + a_0 a|^2 + |a|^2\bigr) |t_1|^{2m} + \bigl(1 + |a_0|^2\bigr)|t_2|^2 \\[2pt]
&\quad + 2\operatorname{Re}\Bigl\{\bigl[(b + a_0 a)\overline{a_0} + a\bigr] t_1^m \overline{t_2}\Bigr\}.
\end{split}
\]
If \(a_0 = 0\), there is nothing to prove. Otherwise we choose \(b={a_0}\) and choose \(a\) to be   
\[
a = -b \frac{\overline{a_0}}{1 + |a_0|^2}.
\]
Then \((b + a_0 a)\overline{a_0}+a= 0\), and consequently  
\[
|\widetilde z_1^m + a_0 \widetilde z_2|^2 + |\widetilde z_2|^2
= \bigl(|b + a_0 a|^2 + |a|^2\bigr) |t_1|^{2m}
+ \bigl(1 + |a_0|^2\bigr)|t_2|^2.
\]
Thus,  after a dilation  in $t$ and using the implicit function theorem,  in the \(t\)-coordinates $M$ is expressed as  
\begin{equation*}
\operatorname{Re} t_{n+1} = |t_1|^{2m} + \sum_{j=2}^{n} |t_j|^2 + O\bigl(\sigma(t')^{1+\alpha}\big)+O\big(|\operatorname{Im}t_{n+1}|^{1+\gamma}\bigr).
\end{equation*}
$M$ is thus $h$--extendible at such a $p_0$ because it has only one zero Levi-eigenvalue  \cite{Yu93} with a local $h$--extendible model $D_m$ defined by $$\operatorname{Re} t_{n+1} > |t_1|^{2m} + \sum_{j=2}^{n} |t_j|^2.$$
\end{proof}
\section{Proof of Theorem \ref{Main theorem 3},  Theorem \ref{Main theorem-CSC}, Theorem \ref{Main theorem 2} and Theorem \ref{Main theorem 2-CSC}}
\begin{proof}[Proof of Theorem \ref{Main theorem 2}]
We now assume the hypotheses  in Theorem \ref{Main theorem 2}. Let $p\in \p\Omega$ be a non-strongly pseudoconvex $h$--extendible boundary point and let $D$ be  its local model at $p$. (Notice that $D$ is unbounded by nature.) Then $D$ has a real analytic boundary and its Bergman metric is  K\"ahler--Einstein by  Theorem \ref{thm1-2-7}.  By Theorem \ref{spherical}, 
Theorem \ref{thm1-2-26}(or a result of Mir-Zaitsev (Theorem 1.4, \cite{MZ21})), and  Theorem \ref{thm1-2-23}, $D$ has  a certain  boundary point $p^\ast\in\partial D$ where the local model is  defined by  :
$$D_m:=\{t\in\mathbb C^{n+1}: {\rm Re}~t_{n+1}>|t_1|^{2m}+\sum_{j=2}^{n}|t_j|^2\}$$  with $m\in\mathbb N, m>1$.
Again by Theorem  \ref{thm1-2-7}, $D_m$ is  K\"ahler--Einstein.
Applying the Cayley transformation, $D_m$ is holomorphically equivalent to the following bounded egg domain
$$\mathcal E_m:=\{z\in\mathbb C^{n+1}:|z_{n+1}|^{2m}+\sum_{j=1}^{n}|z_j|^2<1\}.$$
Now, to prove Theorem~\ref{Main theorem 2}, it suffices to apply the following 
\begin{lemma}\label{example}
    For $m>1$ an integer  and $n\geq 1$, the the Bergman metric of $\mathcal E_m$ can not have constant scalar curvature. In particular, the Bergman metric of $\mathcal E_m$ can not be Einstein.
\end{lemma}
\begin{proof}[Proof of  Lemma \ref{example}]
    We denote by  $S$  the scalar curvature of the Bergman metric of $\mathcal E_m$. If $\mathcal E_m$ has a CSC Bergman metric, then $S\equiv -n-1$ on $\mathcal E_m$ by \cite[Corollary 2]{KYu96} as it has limit $-n-1$ at any strongly pseudoconvex point. From \cite {KYu96}, 
    \begin{equation}\label{scalar}
    S(0)=(n+1)(n+2)-4a_0\sum_{j=1}^{n+1}\frac{b_{jj}}{a_j^2}-a_0\sum_{j\neq k}\frac{b_{jk}}{a_ja_k}.
    \end{equation}
    Here, $a_0=\frac{1}{vol(\mathcal E_m)}$, $a_j=\frac{1}{\|z_j\|_{\mathcal E_m}^2}$, $b_{jk}=\frac{1}{\|z_jz_k\|^2_{\mathcal E_m}}$ with $\|\cdot\|_{\mathcal E_m}$ the $L^2$-norm on $\mathcal E_m$.
    By direct calculations, 
    \begin{equation}\label{5-28-a1}
\begin{aligned}
a_0=\frac{m\Gamma\!\left(n +1+ \frac{1}{m}\right)}{\pi^{n+1} \; \Gamma\!\left(\frac{1}{m}\right)},~
a_1=\dots=a_{n}=\frac{m\Gamma\!\left(n +2+ \frac{1}{m}\right)}{\pi^{n+1} \; \Gamma\!\left(\frac{1}{m}\right)},~ 
a_{n+1}=\frac{m\;\Gamma\!\left(n+1+\frac{2}{m}\right)}{\pi^{n+1} \; \Gamma\!\left(\frac{2}{m}\right)}
\end{aligned}
\end{equation}
and 
\begin{equation}\label{5-28-a2}
\begin{split}
&b_{jj}=\dfrac{m}{2\pi^{n+1}}\,\dfrac{\Gamma\!\left(n+3+\frac1m\right)}{\Gamma\!\left(\frac1m\right)},~ j=1,\dots,n; ~b_{n+1~ n+1}=\dfrac{m}{\pi^{n+1}}\,\dfrac{\Gamma\!\left(n+1+\frac3m\right)}{\Gamma\!\left(\frac3m\right)},\\
&b_{jk}=\dfrac{m}{\pi^{n+1}}\,\dfrac{\Gamma\!\left(n+3+\frac1m\right)}{\Gamma\!\left(\frac1m\right)},~  j\neq k,\; j,k=1,\dots,n,\\
&b_{j~n+1}=b_{n+1~j}=\dfrac{m}{\pi^{n+1}}\,\dfrac{\Gamma\!\left(n+2+\frac2m\right)}{\Gamma\!\left(\frac2m\right)},  ~j=1,\dots,n
\end{split}
 \end{equation}   
Substituting (\ref{5-28-a1}) and (\ref{5-28-a2}) to (\ref{scalar}) we have
\begin{equation}
S(0)=2-\frac{n(n+1+\frac 2m)}{n+1+\frac 1m}
-4\,\frac{\Gamma\!\left(n+1+\frac1m\right)\,\
\Gamma\!\left(n+1+\frac3m\right)\Gamma\!\left(\frac2m\right)^2\,}{\Gamma\!\left(\frac1m\right)\,\Gamma\!\left(\frac3m\right)\,
\Gamma\!\left(n+1+\frac2m\right)^2}.
\end{equation}
By a direct induction argument, beginning with \(n=2\), we obtain \(S(0)<-n-1\), which yields a contradiction.

\end{proof}

This completes the proof of Theorem \ref{Main theorem 2}. 
\end{proof}
We next prove the following theorem, from  which the proof of Theorem~\ref{Main theorem 3} follows.

\begin{theorem}\label{Main theorem 3-1}
Let $\Omega\subset \mathbb C^{n+1}$ $(n\ge 1)$ be a possibly unbounded pseudoconvex domain with real analytic boundary. Suppose that $\partial \Omega$ contains a non-strongly pseudoconvex boundary point but no  nontrivial holomorphic curves. Then the Bergman metric of $\Omega$ cannot be Einstein.
\end{theorem}

\begin{proof}[Proof of Theorem \ref{Main theorem 3-1}]
Suppose that $p_0 \in \partial \Omega$ is a non-strongly pseudoconvex boundary point of $\Omega$. Since $\partial \Omega$ is of D'Angelo finite type at $p_0$, one can find a neighborhood $U$ of $p_0$ and a strongly pseudoconvex boundary point $q \in U \cap \partial \Omega$ sufficiently close to $p_0 $. 
Because the Bergman metric is Einstein, Theorem~\ref{spherical} shows that $\partial \Omega$ is spherical near $q$.  
Let $f \colon V \cap \partial \Omega \to f(V \cap \partial \Omega) \subset \partial \mathbb B^{n+1}$ be a CR diffeomorphism defined in a neighborhood $V$ of $q$ in $\mathbb C^{n+1}$. 
Applying  a theorem of Mir--Zaitsev (Theorem~1.4 of \cite{MZ21}), after shrinking $U$ if necessary, $f$ extends to a holomorphic map $F \colon U \to \mathbb C^{n+1}$ satisfying $F(U \cap \partial \Omega) \subset \partial \mathbb B^{n+1}$. 
Note that $F$ is a local biholomorphism away from a complex analytic hypersurface containing $p_0$. 

By Theorem~\ref{thm1-2-23}, there exists an $h$--extendible point $p^* \in \partial \Omega$ at which the local model domain of $\Omega$ is given by
\[
D_m = \Bigl\{ t \in \mathbb C^{n+1} : \operatorname{Re} t_{n+1} > |t_1|^{2m} + \sum_{j=2}^{n} |t_j|^2 \Bigr\}, \quad m > 1.
\]
Consequently, by Theorem~\ref{thm1-2-7},
\[
J_{\mathbb B^{n+1}} \equiv J_{\Omega} \equiv J_{D_m} \equiv J_{\mathcal E_m}.
\]
This leads to a contradiction by Lemma~\ref{example}. The proof of Theorem~\ref{Main theorem 3-1} is complete.
\end{proof}

\begin{proof}[Proof of Theorem \ref{Main theorem 3}]
Assume that the Bergman metric of $\Omega$ is Einstein. In complex dimension one ($n=0)$, this means that the Bergman metric has constant sectional curvature and hence $\Omega$ is biholomorphic to the unit disk by a classical theorem of Lu \cite{Lu66}. The case $n=1$ is also covered by Savale--Xiao \cite{SX25}. 

Assume now that $n+1 \ge 3$. 
Since $\Omega$ is bounded and has real analytic boundary, it follows from a result of Diederich--Fornaess \cite{DF78} that $\Omega$ is of finite type in the sense of D'Angelo or does not contain any nontrivial holomorphic curves. By Theorem~\ref{Main theorem 3-1}, 
$\Omega$ is a bounded strongly pseudoconvex domain with real analytic boundary. 
A theorem of Huang--Xiao \cite{HX21} then implies that $\Omega$ is biholomorphic to the unit ball of the same dimension. 
This completes the proof of Theorem~\ref{Main theorem 3}.
\end{proof}

\begin{proof}[Proof of Theorem \ref{Main theorem-CSC} and Theorem \ref{Main theorem 2-CSC}]
%Indeed, the K\"ahler--Einstein condition in Theorems \ref{Main theorem 3} and \ref{Main theorem 2} is used only to ensure sphericity at strongly pseudoconvex boundary points. 
We now assume $n\ge 2$. By Theorem \ref{main thm-appendix} in the Appendix, the CSC Bergman assumption implies that the strongly pseudoconvex boundary points are spherical. If $\Omega$ is a bounded real analytic pseudoconvex domain, it follows similarly  from  \cite[Theorem~1.4]{MZ21}, Theorem \ref{thm1-2-23}, Theorem \ref{thm1-2-7-CSC} and Lemma \ref{example} that the CSC Bergman assumption implies that $\Omega$ is actually a strongly pseudoconvex domain. Then $\Omega$ is Bergman--Einstein   by a result of Sha \cite{Sh26} and thus is biholomorphic to the ball. We see the conclusion of Theorem \ref{Main theorem-CSC}. With  Theorem \ref{thm1-2-7-CSC},  Theorem \ref{main thm-appendix} and Lemma \ref{example} at our disposal, the proof of Theorem \ref{Main theorem 2-CSC} is identical  to that of Theorem \ref{Main theorem 2}.

\end{proof}

\vfill
\eject

\section{Appendix: Bergman metric with constant scalar curvature}

\begin{center}
     Xiaojun Huang and Xiaoshan Li
\end{center}

\bigskip
This appendix is a shortened version of the arXiv article \cite{HL26} by the last two authors, which will not be submitted elsewhere.

\bigskip

Let $\Omega\subset\mathbb C^n$ ($n\geq 1$) be a possibly unbounded pseudoconvex domain containing a smooth strongly pseudoconvex boundary point $p\in\partial\Omega$. Then the Bergman metric of \(\Omega\) exists on an open subset \(\Omega^*\) such that \(\Omega\setminus\Omega^*\) is contained in a possibly empty real-analytic hypersurface. Near a strongly pseudoconvex boundary point, its scalar curvature \(S_{\Omega}\) and invariant \(J_{\Omega}\) approach \(-n\) and
$c_n=\frac{(n+1)^n\pi^n}{n!},$ 
respectively.

Starting from the identity
$ 
\log J_\Omega = \log \det G_\Omega - \log K_\Omega,
$ 
and applying \( \partial_{z_k} \partial_{\overline{z_j}} \), then contracting with \( g^{k\bar{j}} \), we obtain that \( S_{\Omega} \) is constant if and only if \( \log J_\Omega \) is harmonic with respect to the Bergman metric, which was first derived by Sha in  (\cite{Sh26}):
\[
\Delta_{g_\Omega} \log J_\Omega (z)
=
\sum_{j,k} g^{k\bar{j}}
\frac{\partial^2}{\partial z^k \,\partial \overline{z^j}} \log J_\Omega=0
\quad \text{on } \Omega^*.
\]

Recall that when  the Bergman metric \( g_\Omega \) has constant scalar curvature on \( \Omega^* \), we say that \( \Omega \) admits a constant scalar curvature ({\bf CSC}) Bergman metric.

In this appendix, we present a   proof of  the following theorem:
\begin{theorem} \label{main thm-appendix}
Let $\Omega
\subset {\mathbb C}^n$ with $n\ge 3$ be a possibly unbounded pseudoconvex domain with $p\in \partial\Omega$ a  smooth strongly pseudoconvex  boundary point. If the Bergman metric of $\Omega$  has  constant scalar curvature,  then $\partial\Omega$ is locally spherical near $p$.
%namely, a small open piece of 
%$\partial\Omega$ is CR-diffeomorphic to an open piece of the boundary of the unit ball $\mathbb{B}^n$. 
\end{theorem}

\begin{proof}[Proof of Theorem \ref{main thm-appendix}]
The proof relies on four main ingredients: Theorem~\ref{Main theorem 1} on the localization of Bergman kernel functions, Christoffer's formula for Chern-Moser-Weyl tensor \cite{Ch81}, a recent elegant formula of Martin for the expansion of \(J_\Omega\) \cite{Ma21}, and several very useful  expansion formulas due to Engliš \cite{Eng08}.

%The proof is similar to that of the corresponding known result (see, e.g., \cite{HX21,HL23}), where the CSC condition is replaced by the slightly stronger K\"ahler--Einstein condition. 
%A new observation here is the use of a recently obtained formula of Martin for the expansion of $J_\Omega$ \cite{Ma21}, in addition to Christoffers's formula for the Bergman expansion \cite{Ch81}. We also use several asymptotic expansion formulas of Engliš \cite{Eng08} to simplify the computation.
%\begin{proof}
Let $G\subset\Omega$ be a small smoothly bounded  strongly pseudoconvex domain with $U\cap G=U\cap \Omega$ for some neighborhood $U$ of the strongly pseudoconvex boundary point $p$ in $\mathbb C^n$. By Theorem \ref{Main theorem 1}, we have
\begin{equation}\label{localization}
    K_\Omega(z, z)=K_G(z, z)+\varphi(z),
\end{equation}
where $\varphi(z)\in C^\infty(U\cap\overline G)$. Let $\rho\in C^\infty(\overline G)$ be a positively signed Fefferman defining function for $G$, namely, $M(\rho)=1+O(\rho^{n+1})$ with the Monge-Ampere operator $M(\rho)$ defined in the following (\ref{MA}). With respect to such a $\rho$, then we have the Fefferman expansion of $K_G(z, z)$ as follows:
\begin{equation}\label{expansion}
    K_G(z, z)=\frac{\phi}{\rho^{n+1}}+\psi\log\rho
\end{equation}
with $\phi, \psi\in C^\infty(\overline G)$ and $\phi=\frac{n!}{\pi^n}+O(\rho^2)$. It follows from (\ref{localization}) and (\ref{expansion}) that 
\begin{equation}\label{AP}
    K_{\Omega}(z, z)=\frac{\phi+\psi\rho^{n+1}\log\rho+\varphi\rho^{n+1}}{\rho^{n+1}}=\frac{\Phi}{\rho^{n+1}},
\end{equation}
where $\Phi=\phi+\psi\rho^{n+1}\log\rho+\varphi\rho^{n+1}\in C^n(U\cap\overline G)$. Moreover, $\Phi=\phi+o(\rho^n)$ on $U\cap\overline G$.
Let $a, b$ be smooth functions on $\overline G$ such that 
\begin{equation*}
    \phi=\frac{n!}{\pi^n}(1+a\rho^2+b\rho^3+O(\rho^4)).
\end{equation*}
Here, $a$ is uniquely determined up to $O(\rho^3)$ near $p$. Then by a result of Christoffers \cite{Ch81}, $a(p)=0$ if and only if $p$ is a CR umbilic point of $\partial G$ in the sense of Chern-Moser \cite{CM74}. Write $u_{\Omega}=(\frac{\pi^n}{n!}K_{\Omega}(z, z))^{\frac{-1}{n+1}}, z\in\Omega^\ast$ and $u_G=(\frac{\pi^n}{n!}K_{G}(z, z))^{\frac{-1}{n+1}}, z\in G$. For any real-valued  $C^2$-smooth function $u$, define
\begin{equation}\label{MA}
{M}(u)=(-1)^n\det\left[\begin{array}{cc}
     u & u_{\overline\beta}  \\
     u_{\alpha}&  u_{\alpha\overline\beta}
\end{array}\right]
\end{equation}
called  the Fefferman-Monge-Ampere operator. By a result of Martin \cite[Page 93]{Ma21}, 
we have
$$M(u_{G})=1-3\frac{n-1}{n+1}a\rho^2+o(\rho^2)~\text{on} ~U\cap \overline G.$$
On the other hand, applying the following Fefferman formula \cite{Ma21} 
\begin{equation}\label{Fefferman formula}
    J_{G}=\frac{(n+1)^n\pi^n}{n!} M(u_{G}):=c_n M(u_{G}) \ \ \hbox{where  } c_n=\frac{(n+1)^n\pi^n}{n!}
\end{equation}
 we have
\begin{equation}
    J_G(z)=c_n-3c_n\frac{n-1}{n+1}a\rho^2+o(\rho^2).
\end{equation}
 
Again by the Fefferman formula, we verify that 
$J_\Omega=J_G+O(\rho^3)$ near $p$. Thus 
\begin{equation}
    J_{\Omega}(z)=c_n-3c_n\frac{n-1}{n+1}a\rho^2+o(\rho^2)~\text{near }p.
\end{equation}

Since $J_\Omega(z)$  is independent of the chosen defining function and two defining functions differ by a positive smooth function in a neighborhood of $p$, we may assume, without loss of generality,  in the following computation that $\rho$ is strongly plurisubharmonic near $p$ to prove $a(p) = 0$.
By (\ref{AP}), we have the following expansion of the Bergman canonical invariant function from \cite{Eng08} (Engliš stated this result for strongly pseudoconvex domains, with (\ref{AP}),  however, his proof carries over without any change to our local setting):

\begin{equation}
\begin{split}
    J_{\Omega}(z)&\sim\sum_{j=0}^\infty(\rho^{n+1}\log\rho)^j\eta_j, \eta_j\in C^\infty(U\cap\overline G)\\
    &    \sim\eta_0+\eta_1\rho^{n+1}\log\rho+\cdots.
    \end{split}
\end{equation}
Here, the sum is in the asymptotic sense, that is, for any $k\in\mathbb N$, the difference $$J_\Omega(z)-\sum_{j=0}^{k-1}(\rho^{n+1}\log\rho)^j\eta_j\in C^{k(n+1)-1}(U\cap\overline G)$$ and vanishes on $U\cap\partial G$ with all its partial derivatives of orders $\leq k(n+1)-1$. 
Then we can write 
\begin{equation}
   \log( J_{\Omega}(z))=\log(c_n)-3\frac{n-1}{n+1}a(z)\rho^2+b(z)\rho^2
\end{equation}
where $b(z)\in C^\infty(U\cap G)\cap C^{1,\frac{1}{2}}(U\cap\overline G)$ and 
\begin{equation}\label{5-24-a1}
b(z)\sim a_1(z)\rho+\rho^{-2}\sum_{j=1}^\infty \tilde \eta_{j}(\rho^{n+1}\log\rho)^j,
\end{equation}
with $\tilde \eta_k(z)\in C^\infty(U\cap\overline G)$ for $k\ge 1.$

Since 
 $\Delta_{g_\Omega}\log J_{\Omega}=0$,  we have 
\begin{equation}\label{5-23-a3}
\sum_{i, j=1}^ng^{\overline ji}\frac{\partial^2}{\partial z_i\partial\overline z_j}[3a\frac{n-1}{n+1}\rho^2- b\rho^2]\equiv 0~\text{on} ~U\cap G.
\end{equation}

By direct calculation,
\begin{equation}\label{5-23-a1}
\begin{split}
     g^{\overline ji}\frac{\partial^2}{\partial z_i\partial\overline z_j}(a\rho^2)=&\rho^2 g^{\overline j i}\frac{\partial^2 a}{\partial z_i\partial\overline z_j}+2\rho g^{\overline j i}\left(\frac{\partial a}{\partial z_i}\frac{\partial\rho}{\partial \overline z_j}+\frac{\partial a}{\partial \overline z_j}\frac{\partial \rho}{\partial z_i}\right)+\\
    &2a g^{\overline j i}\frac{\partial \rho}{\partial\overline z_j}\frac{\partial\rho}{\partial z_i}+2a\rho g^{\overline ji}\frac{\partial^2 \rho}{\partial z_i\partial\overline z_j}.
    \end{split}
\end{equation}
Since we can write $K_{\Omega}=\rho^{-(n+1)}[\widetilde \eta_0+(\rho^{n+1}\log\rho)\widetilde\eta_1],$
where $\widetilde\eta_0=\phi+\varphi\rho^{n+1}$ and $\widetilde\eta_1=\psi$, from  \cite[Theorem 3]{Eng08}, we have 
\begin{equation}
    \rho^{-1} g^{\overline j i}\in C^n(U\cap\overline G).
\end{equation}
Moreover, from \cite[Section 3]{Eng08} we have in  a small neighborhood $U_p\subset {\mathbb C}^n$ of $p$
\begin{equation}\label{5-23-a2}
\begin{split}
    &\rho^{-1}g^{\overline j i}=\rho^{-1}[\log \rho]^{\overline j l}H^i_l, H^i_l\in C^n(U_p\cap\overline G), H^i_l|_{\partial G\cap U_p}=-\frac{1}{n+1}\delta^i_l.\\
    &\frac{1}{\rho^2}[\log\rho]^{\overline j i}\rho_i, ~~\frac{1}{\rho^2}[\log\rho]^{\overline j i}\rho_{\overline j}\in C^\infty(\overline G\cap U_p),~[\log \rho]^{\overline j i}\rho_{\overline j}\rho_i=\rho [\log \rho]^{\overline j i}\rho_{i\overline j}-n\rho^2~\text{on}~U_p\cap\overline G.\\
    &\lim_{z\rightarrow p} \frac{1}{\rho^2}[\log \rho]^{\overline j i}\rho_{\overline j}\rho_i=-1.
    \end{split}
\end{equation}
Replacing $a(z)$ by $b(z)$ from (\ref{5-24-a1}) in (\ref{5-23-a1}) and from (\ref{5-23-a2}) we have that
\begin{equation}
    \lim_{z\rightarrow p}\frac{1}{\rho^2}\sum_{i, j=1}^n g^{\overline ji}\frac{\partial^2}{\partial z_i\partial\overline z_j}(b\rho^2)=0.
\end{equation}
It follows from (\ref{5-23-a3}), (\ref{5-23-a1}) and (\ref{5-23-a2}) that 
\begin{equation}
    0=\lim_{z\rightarrow p}\frac{1}{\rho^2}\sum_{i, j=1}^n g^{\overline ji}\frac{\partial^2}{\partial z_i\partial\overline z_j}(a\rho^2)=2a(p)(n-2).
\end{equation}
Since $n\geq 3$, we have $a(p)=0$ and thus an arbitrary given strongly pseudoconvex point $p$ is a CR umbilic point of $\partial\Omega$.
It follows immediately from the Chern-Moser Lemma \cite{CM74} that $\partial\Omega$ is spherical near $p$.
This completes the proof of Theorem  \ref{main thm-appendix}.

\end{proof}

{\bf Statements and Declarations}: The authors have no financial or non-financial
interests that are related to this work. No data was collected
during the course of this work and  no AI was used in this work.

\end{document}